\documentclass[a4paper,12pt]{article}

\usepackage[english]{babel} 
\usepackage{amsmath,amsfonts,amssymb,amsthm,bbm,graphics,epsfig,psfrag,epstopdf,textcomp,stmaryrd}
\usepackage[active]{srcltx}
\usepackage{paralist,subfigure,float}
\usepackage{color,soul}
\usepackage[dvipsnames]{xcolor}
\usepackage{hyperref} 
\usepackage{autonum} 

\setstcolor{red}

\newtheorem{theorem}{Theorem}[section]
\newtheorem{corollary}[theorem]{Corollary}
\newtheorem{lemma}[theorem]{Lemma}
\newtheorem{proposition}[theorem]{Proposition}
\newtheorem{definition}[theorem]{Definition}
\newtheorem{remark}[theorem]{Remark}

\newtheorem{conjecture}[theorem]{Conjecture}

\def\N{\mathbb{N}}
\def\Z{\mathbb{Z}}

\def\R{\mathbb{R}}
\def\epsilon{\varepsilon}
\let\e=\varepsilon
\let\vp=\varphi
\let\vt=\vartheta
\let\t=\widetilde
\let\ol=\overline

\let\.=\cdot
\let\0=\emptyset
\let\mc=\mathcal

\def\O{\Omega}

\def\hat{\widehat}
\def\tilde{\widetilde}
\def\eN{\mathrm{e}_N}

\DeclareMathOperator{\dist}{dist}

\DeclareMathOperator\supp{supp}

\DeclareMathOperator*{\essinf}{ess{\,}inf}

\def\thm#1{Theorem~\ref{thm:#1}}
\def\seq#1{(#1_n)_{n\in\N}}

\def\as{\quad\text{as }\;}

\def\tilde{\widetilde}
\def\1{\mathbbm{1}}
\def\Sph{\mathbb{S}^{N-1}}

\def\aos{asymptotic one-dimensional symmetry}

\newenvironment{formula}[1]{\begin{equation}\label{#1}}{\end{equation}\noindent}
\def\Fi#1{\begin{formula}{#1}}
\def\Ff{\end{formula}\noindent}

\newcommand{\be}{\begin{equation}}
\newcommand{\ee}{\end{equation}}
\newcommand{\baa}{\begin{array}}
\newcommand{\eaa}{\end{array}}
\newcommand{\ba}{\begin{eqnarray}}
\newcommand{\ea}{\end{eqnarray}}

\numberwithin{equation}{section}

\begin{document}
\date{}
\title{\bf{Asymptotic one-dimensional symmetry for the Fisher-KPP equation}}
\author{Fran\c cois Hamel$^{\hbox{\small{ a}}}$ and Luca Rossi$^{\hbox{\small{ b,c }}}$\thanks{This work has received funding from Excellence Initiative of Aix-Marseille Universit\'e~-~A*MIDEX, a French ``Investissements d'Avenir'' programme, and from the French ANR RESISTE (ANR-18-CE45-0019) project. The first author acknowledges support of the Institut Henri Poincaré (UAR 839 CNRS-Sorbonne Universit\'e), LabEx CARMIN (ANR-10-LABX-59-01), and Universit\`a degli Studi di Roma La Sapienza, where he was Sapienza Visiting Professor and where part of this work was done.}\\
\\
\footnotesize{$^{\hbox{a }}$Aix Marseille Univ, CNRS, Centrale Marseille, I2M, Marseille, France}\\
\footnotesize{$^{\hbox{b }}$SAPIENZA Univ Roma, Istituto ``G.~Castelnuovo'', Roma, Italy}\\
\footnotesize{$^{\hbox{c }}$CNRS, EHESS, CAMS, Paris, France}\\
}
\maketitle

\begin{abstract}
\noindent{}Let $u$ be a solution of the Fisher-KPP equation
\[
\partial_t u=\Delta u+f(u),\quad t>0,\ x\in\R^N.
\]
We address the following question: does~$u$ become locally planar as $t\to+\infty$~? Namely, does~$u(t_n,x_n+\.)$ converge locally uniformly, up to subsequences, towards a one-dimensional function, 
for any sequence $((t_n,x_n))_{n\in\N}$ in $(0,+\infty)\times\R^N$ such that $t_n\to+\infty$ as $n\to+\infty$~? 
This question is in the spirit of a conjecture of De Giorgi for stationary solutions of Allen-Cahn equations. 
The answer depends on the initial datum $u_0$ of~$u$. It is known to be affirmative when the support of~$u_0$ is bounded or when it lies between two parallel half-spaces. Instead, the answer is negative when the support of~$u_0$ is~``V-shaped''. We prove here that $u$ is asymptotically locally planar when the support of~$u_0$ is a convex set (satisfying in addition a uniform interior ball condition), or, more generally, when it is at finite Hausdorff distance from a convex set. We actually derive the result under an even more general geometric hypothesis on the support of~$u_0$. We recover in particular the aforementioned results known in the literature. We further characterize the set of directions in which~$u$ is asymptotically locally planar, and we show that the asymptotic profiles are monotone. Our results apply in particular when the support of~$u_0$ is the~subgraph of a function with vanishing global mean.
\end{abstract}
	
	
		
\section{Introduction}\label{sec:intro}
	
In this paper, we are interested in the large time description of solutions of the Fisher-KPP reaction-diffusion equation
\Fi{homo}
\partial_t u=\Delta u+f(u),\quad t>0,\ x\in\R^N,
\Ff
with $N\ge2$. The Fisher-KPP condition of~\cite{F,KPP} is
\be\label{fkpp}
\begin{cases}
f(0)=f(1)=0,\\
f(s)>0\hbox{ for all }s\in(0,1),\\
\displaystyle s\mapsto\frac{f(s)}{s}\hbox{ is nonincreasing in $(0,1]$}.
\end{cases}
\ee
As long as the regularity is concerned, we assume that~$f\in C^1([0,1])$. These assumptions on $f$ will always be  understood to hold.

We consider the Cauchy problem associated with~\eqref{homo}. The initial condition $u(0,\.)=u_0$ is assumed to be a characteristic function $\1_U$ of a measurable set~$U\subset\R^N$, i.e.
\be\label{defu0}
u_0(x)=\left\{\baa{ll}1&\hbox{if }x\in U,\vspace{3pt}\\ 0&\hbox{if }x\in\R^N\!\setminus\!U.\eaa\right.
\ee
This Cauchy problem is well posed and, given $u_0$, there is a unique bounded classical solution $u$ of~\eqref{homo} such that~$u(t,\cdot)\to u_0$ as~$t\to0^+$ in~$L^1_{loc}(\R^N)$. Furthermore, $0\le u(t,x)\le 1$ for all $t\ge0$ and $x\in\R^N$, from the maximum principle. For mathematical convenience, we extend~$f$ by~$0$ in~$\R\setminus[0,1]$, and the extended function, still denoted $f$, is then Lipschitz continuous in $\R$.

Instead of initial conditions $u_0=\1_U$, we could also have considered multiples~$\alpha\1_U$ of characteristic functions, with $\alpha>0$, or even other more general initial conditions $0\leq u_0\leq 1$ for which the upper level set $\{x\in\R^N:u_0(x)\geq h\}$ is at bounded Hausdorff distance from the support of $u_0$, for some $h\in(0,1)$ (see Section~\ref{secgeneral} below). But we preferred to keep the assumption~$u_0=\1_U$ for the sake of simplicity of the presentation of the statements.\\

The goal of the paper is to understand whether, and under which condition on the initial datum, the solution of~\eqref{homo} eventually becomes locally planar as time goes on. To express this property in a rigorous way, we consider the notion of the {\em $\O$-limit set} of a given bounded function $u:\R^+\times\R^N\to\R$, which is defined as follows: 
\be\label{def:Omega}\begin{array}{ll}
\hspace{-8pt}\Omega(u)\!:=\!\big\{&\!\!\!\!\!\psi\in L^\infty(\R^N)\, :\,   u(t_n,x_n+\.)\to\psi\text{ in $L^\infty_{loc}(\R^N)$ as $n\to+\infty$,}\\
& \!\!\!\!\!\text{for some sequences $(t_n)_{n\in\N}$ in $\R^+$ diverging to $+\infty$, and $(x_n)_{n\in\N}$ in $\R^N$}\big\}.
\end{array}
\ee
Roughly speaking, the $\O$-limit set contains all possible asymptotic profiles of the function as $t\to+\infty$. Notice that, for any bounded solution $u$ of~\eqref{homo}, the set $\Omega(u)$ is not empty and is included in $C^2(\R^N)$, from standard parabolic estimates. We say that $u$ is {\em asymptotically locally planar} if every $\psi\in\Omega(u)$ is {\it one-dimensional}, that is, if~$\psi(x)\equiv\Psi(x\.e)$ for all $x\in\R^N$, for some $\Psi\in C^2(\R)$ and~$e\in\Sph$, where $\Sph:=\{x\in\R^N:|x|=1\}$, $|\cdot|$ denotes the Euclidean norm in $\R^N$, and ``$\cdot$" denotes the Euclidean scalar product in $\R^N$. Furthermore, we say that $\psi$ is {\it one-dimensional and $($strictly$)$ monotone} if $\psi$ is as above with $\Psi$ (strictly) monotone.

This property reclaims the De Giorgi conjecture about bounded solutions of the Allen-Cahn equation (that is, bounded stationary solutions of the reaction-diffusion equation $\Delta u+u(1-u)(u-1/2)=0$ in $\R^N$, obtained after a change of unknown from the original Allen-Cahn equation), see~\cite{dGconj}. 
\\

Let us review the results in the literature about the asymptotic one-dimensional symmetry. Consider, as before, solutions emerging from indicator functions of a set~$U$. First, the asymptotic one-dimensional symmetry is known to hold when $U$ is bounded, as a consequence of~\cite{J}. The case of unbounded sets $U$ has been much less studied in the literature. However, the asymptotic one-dimensional symmetry is known to hold when~$U$ is the subgraph of a bounded function, by~\cite{BH1,B,HNRR,L,U1}. Conversely, the property fails when~$U$ is ``V-shaped'', i.e. when $U$ is the union of two half-spaces with non-parallel boundaries, as follows from the methods developed in~\cite{HN}, see Proposition~\ref{proV} below for further details. The properties listed in this paragraph are known to hold for other types of function $f$ as well, see~\cite{BH1,FM,HMR1,HMR2,MN,MNT,NT,RR1} and Section~\ref{secgeneral}.

A possible interpretation of these results is that the asymptotic one-dimensional symmetry holds provided $U$ is ``not too far'' from a convex set. This is indeed what we will~show.


\section{Statement of the main results}\label{secaos}

In order to state our main results, we define the notion of positive-distance-interior of a set $U\subset\R^N$ as
$$U_\delta:=\big\{x\in U:\dist(x,\partial U)\ge\delta\big\},\quad\delta>0,$$
where $\dist(x,A):=\inf\big\{|x-y|:y\in A\big\}$ for a set $A\subset\R^N$, with the convention $\dist(x,A)=+\infty$ if $A=\emptyset$. Throughout the paper, we denote
$$B_r(x):=\big\{y\in\R^N:|x-y|<r\big\}\ \hbox{ and }\ B_r:=B_r(0),$$
for any $x\in\R^N$ and $r>0$. We will also make use of the Hausdorff distance between subsets of $\R^N$, which is defined, for $A,B\subset\R^N$, by  
$$d_{\mc{H}}(A,B):=\max\Big(\sup_{x\in A}\dist(x,B),\sup_{y\in B}\dist(y,A)\Big),$$
with the conventions that
$$d_{\mc{H}}(A,\emptyset)=d_{\mc{H}}(\emptyset,A)=+\infty\hbox{ \ if $A\neq\emptyset\ $ \ and \ $\ d_{\mc{H}}(\emptyset,\emptyset)=0$}.$$

\begin{theorem}\label{thm:DG}
Let $u$ be the solution of~\eqref{homo} with an initial datum $u_0=\1_{U}$ such that $U\subset\R^N$ satisfies 
\Fi{UUdelta}
\exists\,\delta>0,\quad d_{\mc{H}}(U,U_\delta)<+\infty.
\Ff 
Assume moreover that $U$ is convex or {\em nearly convex}, that is, that there exists a convex set~$U'\subset\R^N$ satisfying $d_{\mc{H}}(U,U')<+\infty$. Then any function in $\Omega(u)$ is one-dimensional and, in addition, it is either constant or strictly monotone.
\end{theorem}

Theorem~\ref{thm:DG} extends the known results about the asymptotic one-dimensional symmetry for the Fisher-KPP equation that, we recall, have been established under the assumption that $U$ is bounded \cite{J}, or it is the subgraph of a bounded function~\cite{BH1,B,HNRR,L,U1}. 

Condition~\eqref{UUdelta} means that there exists some $R>0$ such that, for any~$x\in U$, there is a ball $B_\delta(x_0)$ contained in $U$ of radius $\delta$ and centered at a point $x_0$ such that $|x-x_0|<R$. It is fulfilled in particular if~$U$ satisfies a uniform interior ball condition. It is not hard to see that, in dimension $N=2$, for a convex set $U$, property~\eqref{UUdelta} is equivalent to require that~$U$ has nonempty interior. The role of assumption~\eqref{UUdelta} is cutting off regions of $U$ which play a negligible role in the large-time behavior of the solution of the Cauchy problem. This assumption is necessary, otherwise one could take a set $U'$ for which the asymptotic one-dimensional symmetry fails (for instance V-shaped), then consider the set $U:=U'\,\cup\,\bigcup_{k\in\Z^N}B_{e^{-|k|^2}}(k)$: the set $U$ is nearly convex, being at finite Hausdorff distance from the convex set~$\R^N$, but it does not satisfy~\eqref{UUdelta} and the asymptotic one-dimensional symmetry fails for $U$, see Proposition~\ref{proV2} below for further details.

Let us describe the idea of the proof of \thm{DG}, in relation with the previous proofs in the literature. First, when $U$ is the subgraph of a bounded function, the one-dimensional symmetry is derived combining~\cite{B} or~\cite{HNRR,L,U1} with the Liouville-type result of~\cite{BH1}, which asserts that an entire solution trapped between two translations of a planar traveling front is necessary a planar traveling front --~hence one-dimensional at every time. The Liouville result is proved using the sliding method. The proof of~\cite{J} concerning the case where $U$ is bounded is much quicker. It relies on the same reflection argument as in the moving plane method. This very elegant proof works for arbitrary nonlinear terms~$f$, but unfortunately fails as soon as~$U$ is unbounded. To circumvent this difficulty, we develop an argument that allows us to extend the technique of~\cite{J} even when~$U$ is unbounded. This is the main technical contribution of the present paper. The idea is to approximate the solution through a suitable truncation of the initial support~$U$. The choice of the truncation cannot be made once and for all, but it rather depends on the time at which we want to approximate the solution. In order to control the approximation error, we exploit a new family of retracting supersolutions obtained as a superposition of a large number of traveling fronts, see Lemma~\ref{lem:super} below.

\vspace{-7pt}
\begin{figure}[H]
	\begin{center}
		\includegraphics[height=6cm]{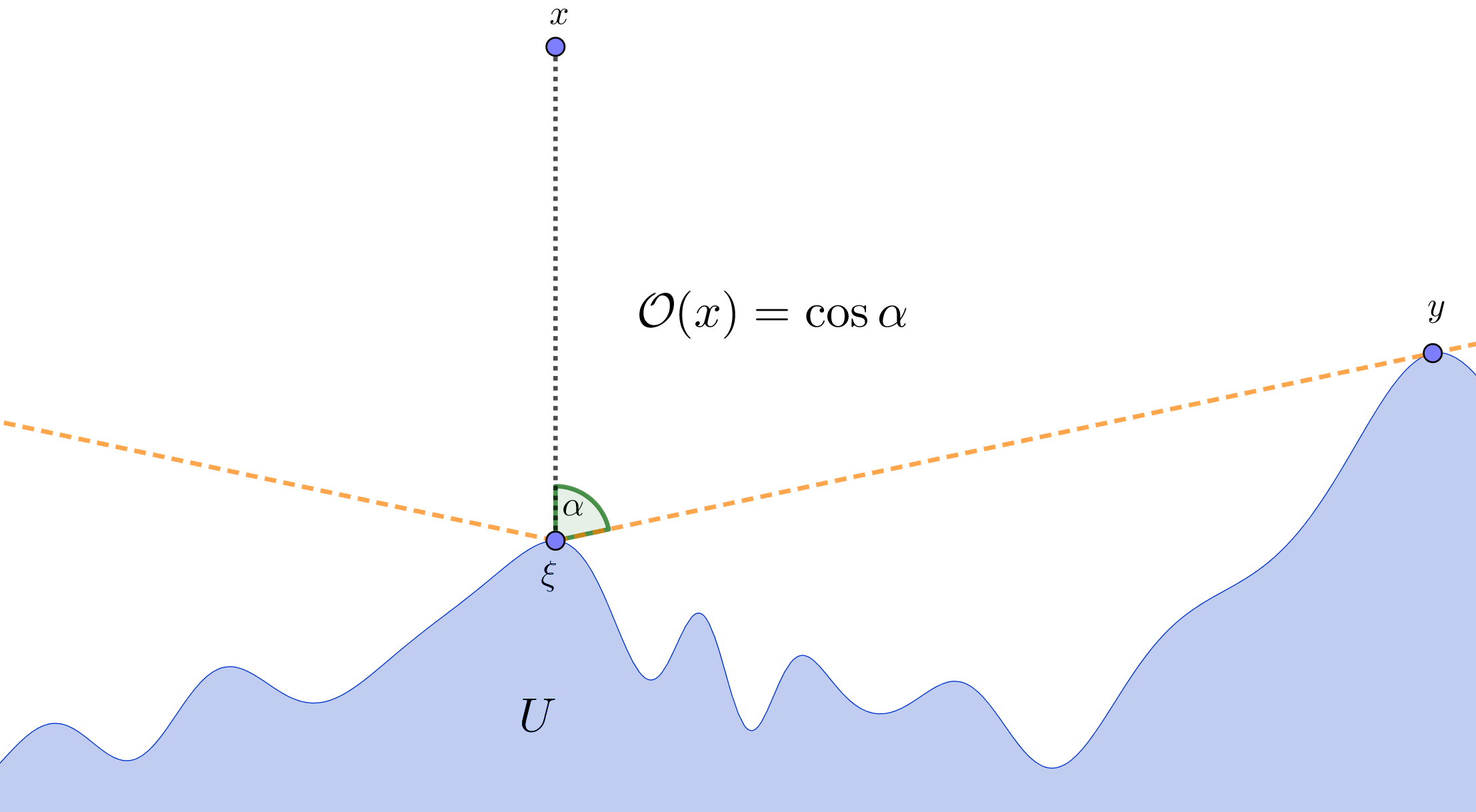}
		\caption{The definition of the {\em opening} function~$\mc{O}$.}
		\label{fig:O}
	\end{center}
	\vspace{-7pt}
\end{figure}
\vspace{-7pt}
As a matter of fact, the convex-proximity assumption on $U$ in \thm{DG} is a very special case of a geometric hypothesis under which we prove the one-dimensional symmetry, that we now introduce. For a given nonempty set $U\subset\R^N$ and a given point $x\in\R^N$, we let $\pi_x$ denote the set of orthogonal projections of~$x$ onto $\ol U$, i.e.,
\be\label{defpix}
\pi_x:=\big\{\xi\in\ol U: |x-\xi|=\dist(x,U)\big\},
\ee
and, for $x\notin\ol U$, we define the {\em opening} function as follows:
\Fi{opening}
\mc{O}(x):=\sup_{\xi\in\pi_x,\,y\in U\setminus\{\xi\}}\,\frac{x-\xi}{|x-\xi|}\.\frac{y-\xi}{|y-\xi|},
\Ff
with the convention that $\mc{O}(x)=-\infty$ if $U=\emptyset$ or $U$ is a singleton (otherwise $-1\leq\mc{O}(x)\leq1$). Namely, when $\mc{O}(x)\neq-\infty$, one has $\mc{O}(x)=\cos\alpha$, where $\alpha$ is the infimum among all~$\xi\in\pi_x$ of half the opening of the largest exterior cone to $U$ at~$\xi$ having axis~$x-\xi$, see Figure~\ref{fig:O}. 

Here is our most general asymptotic symmetry result.

\begin{theorem}\label{thm:DGgeneral}
Let $u$ be a solution of~\eqref{homo} with an initial datum $u_0=\1_{U}$ such that $U\subset\R^N$ satisfies~\eqref{UUdelta} and moreover
\Fi{ballcone}
\lim_{R\to+\infty}\bigg(\,\sup_{x\in\R^N,\,\dist(x,U)=R}\mc{O}(x)\bigg)\leq 0.
\Ff		
Then any function in $\Omega(u)$ is one-dimensional and, in addition, it is either constant or strictly monotone.
\end{theorem}

It is understood that the left-hand side in condition~\eqref{ballcone} is equal to $-\infty$ (hence the condition is fulfilled) if $\sup_{x\in\R^N}\dist(x,U)<+\infty$ (and indeed in such case the asymptotic one-dimensional symmetry trivially holds because condition~\eqref{UUdelta} yields that $u(t,x)\to1$ uniformly in $x\in\R^N$ as~$t\to+\infty$, see Proposition~\ref{pro:uniformspreading} below). We remark that the limit in~\eqref{ballcone} always exists, because the involved quantity is nonincreasing with respect to~$R$, see Lemma \ref{lem:ballcone} below.

The optimality of hypotheses~\eqref{UUdelta} and~\eqref{ballcone} is discussed in Section~\ref{sec54} below. Hypothesis~\eqref{ballcone} means that the angle $\alpha$ in Figure~\ref{fig:O} tends to a value larger than or equal to~$\pi/2$ (which means that the exterior cone contains a half-space) as $\dist(x,U)\to+\infty$. Theorem~\ref{thm:DGgeneral} yields Theorem~\ref{thm:DG} because, firstly, convex sets satisfy $\mc{O}(x)\leq0$ for every $x\notin\ol U$ (actually, they are characterized by such condition in the class of closed sets) and, secondly, if \eqref{ballcone} holds for a given set, then it holds true for any set at finite Hausdorff distance from it, as stated by~Lemma \ref{lem:ballcone}. However, the class of sets satisfying~\eqref{ballcone} is wider. It contains for instance the subgraphs of functions with {\em vanishing global mean},~i.e.,
\Fi{Usub}
U=\big\{x=(x',x_N)\in\R^{N-1}\times\R \ : \ x_N\le\gamma(x')\big\},
\Ff
with $\gamma\in L^\infty_{loc}(\R^{N-1})$ such that
\Fi{VGM}
\frac{\gamma(x')-\gamma(y')}{|x'-y'|}\longrightarrow 0\as|x'-y'|\to+\infty.
\Ff
As a matter of fact, when $U$ is given by \eqref{Usub}-\eqref{VGM}, we derive a more precise characterization of the
$\Omega$-limit set, see Corollary~\ref{cor:DGVGM} below. 

Theorems~\ref{thm:DG} and~\ref{thm:DGgeneral} are concerned with locally uniform convergence properties along sequences of times $(t_n)_{n\in\N}$ diverging to $+\infty$ and sequences of points $(x_n)_{n\in\N}$. We now assert an asymptotic property which is satisfied {\em uniformly} in $\R^N$. It is expressed in terms of the eigenvalues of the Hessian matrices $D^2u(t,x)$ (with respect to the $x$ variables). For a symmetric real-valued matrix $A$ of size $N\times N$, let $\lambda_1(A)\le\cdots\le\lambda_N(A)$ denote its eigenvalues, and let
$$\sigma_k(A):=\sum_{1\le j_1<\cdots<j_k\le N}\lambda_{j_1}(A)\times\cdots\times\lambda_{j_k}(A),\ \ \ \ 1\le k\le N,$$
be the elementary symmetric polynomials of eigenvalues of $A$ ($\sigma_k(D^2u(t,x))$ is also called $k$-Hessian).

\begin{theorem}\label{thm:global}
Let $u$ be as in Theorem~$\ref{thm:DGgeneral}$. Then,
$$\forall\,2\le k\le N,\quad \sigma_k(D^2u(t,x))\to0\ \hbox{ as $t\to+\infty$ uniformly in $x\in\R^N$.}$$
\end{theorem}

The proof of Theorem~\ref{thm:global} is based on the asymptotic local one-dimensional symmetry given in Theorem~\ref{thm:DGgeneral}, and on standard parabolic estimates. We point out that, if $\psi:\R^N\to\R$ is of class $C^2(\R^N)$ and one-dimensional, then $\sigma_k(D^2\psi(x))=0$ for all $2\le k\le N$ and $x\in\R^N$, since the quantities $\sigma_k(D^2\psi(x))$ involve sums of products of at least two eigenvalues of $D^2\psi(x)$ (but $\sigma_1(D^2\psi(x))\neq0$ in general). However, the converse property is immediately not true (for instance, the function $\psi:(x_1,x_2)\mapsto x_1^2+x_2$ satisfies $\sigma_2(D^2\psi(x_1,x_2))=0$ for all $(x_1,x_2)\in\R^2$, but it is not one-dimensional).

Once the asymptotic one-dimensional symmetry and monotonicity properties are established, it is natural to ask what are the directions in which the solution actually becomes locally one-dimensional. Namely, we investigate the~set
\Fi{E}
\begin{split}
\mc{E}:=\big\{\, & e\in\Sph\, :\, \exists\,\psi\in\O(u)\text{ such that $\,\psi(x)\equiv\Psi(x\.e)$}\\
& \text{for some strictly decreasing function $\Psi\in C^2(\R)$}\big\}.
\end{split}
\Ff
Under the assumptions of Theorems~\ref{thm:DG} or~\ref{thm:DGgeneral}, the set $\mc{E}$ is then the set of the directions of  decreasing monotonicity of all non-constant elements of $\Omega(u)$.\footnote{By the direction of decreasing monotonicity of a --necessarily one-dimensional by Theorems~\ref{thm:DG} or~\ref{thm:DGgeneral}-- non-constant function $\psi\in\Omega(u)$, we mean the unique $e\in\Sph$ such that $\psi(x)=\Psi(x\cdot e)$ for all $x\in\R^N$, with $\Psi$ decreasing.} Observe that the constant functions $\psi$ are excluded in the above definition, which is necessary because they are one-dimensional in every direction. Thus, a direction~$e$ belongs to~$\mc{E}$ only if, along diverging sequences of times, the solution flattens in the directions orthogonal to~$e$ but not in the direction $e$, along some sequence of points. We characterize the set $\mc{E}$ in terms of the initial support $U$.

\begin{theorem}\label{thm:E}
Let $u$ be as in Theorem~$\ref{thm:DGgeneral}$. Then the set $\mc{E}$ defined in~\eqref{E} is given by
\[
\begin{split}
\mc{E}=\Big\{ \, & e\in\Sph\ :\  \displaystyle\frac{x_n-\xi_n}{|x_n-\xi_n|}\to e\ \hbox{as $n\to+\infty$, for some sequences $(x_n)_{n\in\N}$, $\seq{\xi}$ in $\R^N$}\\
&\text{such that $\,\dist(x_n,U)\to+\infty$ as $n\to+\infty$ and $\xi_n\in \pi_{x_n}$ for all $n\in\N$}\Big\}.
\end{split}
\]
In particular, $\mc{E}=\emptyset$ if and only if $U$ is relatively dense in~$\R^N$ or $U=\emptyset$.
\end{theorem}

We remark that, without the assumption~\eqref{UUdelta}, the last statement of \thm{E} may~fail. Indeed, if $U=\{0\}$ then $u(t,x)\equiv0$ for all $t>0$, $x\in\R^N$, hence $\mc{E}=\emptyset$, but~$U\neq\emptyset$ is not relatively dense in $\R^N$.

When $U$ is bounded with non-empty interior, it follows from Theorem~\ref{thm:E} that $\mc{E}=\Sph$. On the one hand, this conclusion gives an additional property --namely the strict monotonicity-- with respect to the result contained in~\cite{J}. On the other hand, still when $U$ is bounded, the same conclusion is also a consequence of~\cite{D,RRR}, where it is proved by a completely different argument. The characterization of the directions of asymptotic strict monotonicity in the case of unbounded sets $U$ is more involved. The proof of Theorem~\ref{thm:E} is based on an argument by contradiction and on the acceleration of the solutions when they become less and less steep.

\thm{E} implies that if~$U$~is of class~$C^1$ then~$\mc{E}$ is contained in the closure of the set of the outward unit normal vectors~to~$U$. If $U$ is convex then $\mc{E}$ coincides with the closure of the set of outward unit normal vectors to all half-spaces containing $U$. When $U$ is the subgraph of a function $\gamma$ with vanishing global mean, i.e.~satisfying~\eqref{VGM}, then we show that~$\mc{E}=\{\mathrm{e}_N\}$, where $\mathrm{e}_N:=(0,\cdots,0,1)$. Namely, in such a case we have the following.

\begin{corollary}\label{cor:DGVGM}
Let $u$ be the solution of~\eqref{homo} with an initial datum $u_0=\1_{U}$, where $U$ is given by~\eqref{Usub} with $\gamma\in L^\infty_{loc}(\R^{N-1})$ satisfying~\eqref{VGM}. Then any function $\psi\in\Omega(u)$ is of the form $\psi(x',x_N)\equiv\Psi(x_N)$ for all $(x',x_N)\in\R^{N-1}\times\R$, with $\Psi\in C^2(\R)$ either constant or strictly decreasing. Moreover, it holds that $\mc{E}=\{\mathrm{e}_N\}$.
\end{corollary}

Since, by parabolic estimates, the convergence in the definition~\eqref{def:Omega} of the $\O$-limit set holds true in $C^2_{loc}(\R^N)$, up to subsequences, Corollary~\ref{cor:DGVGM} implies that
$$\nabla_{\!x'} u(t,x',x_N)\to0\as t\to+\infty,\ \text{ uniformly with respect to $(x',x_N)\in\R^{N-1}\times\R$}.$$
A way to interpret this result is that the oscillations of the initial datum are ``damped'' as time goes~on through some kind of averaging process. We point out that Corollary~\ref{cor:DGVGM} does not imply the existence of a function $\Psi:\R^+\times\R\to\R$ such that $u(t,x',x_N)-\Psi(t,x_N)\to0$ as $t\to+\infty$ uniformly in $(x',x_N)\in\R^{N-1}\times\R$, and indeed such a function $\Psi$ does not exist in general (as shown in~\cite{RR2} when $N=2$ and the limits $\lim_{x'\to\pm\infty}\gamma(x')$ exist but do not coincide). Condition~\eqref{VGM} is satisfied in particular when $\gamma$ is bounded, and in such a case the conclusion of Corollary~\ref{cor:DGVGM} can also be deduced from~\cite{BH1,B,HNRR,L,U1}.

It is possible to relax the uniform mean condition~\eqref{VGM} of $\gamma$ in Corollary~\ref{cor:DGVGM}, at the price of restricting the $\O$-limit set. With this regard, we will derive a {\em directional} asymptotic symmetry result, \thm{DGe} below.

\hfill\break
\noindent{\bf{Outline of the paper.}}  The rest of the paper is organized as follows. In Section~\ref{sec:preliminary}, we show a general uniform spreading result, which is itself based especially on the construction of retracting super-solutions as finite sums of planar fronts. Section~\ref{sec:a1D} is the central section, devoted to the proofs of Theorems~\ref{thm:DG},~\ref{thm:DGgeneral} and~\ref{thm:global} on the asymptotic one-dimensional symmetry of the solutions under the general hypotheses~\eqref{UUdelta} and~\eqref{ballcone}. Some counterexamples to the main results, when at least one of these assumptions is not fulfilled, are also shown. Section~\ref{sec5} contains the proof of Theorem~\ref{thm:E} on the set of directions of asymptotic strict monotonicity. The case of a subgraph with vanishing global mean is dealt with in Section~\ref{sec6}, where Corollary~\ref{cor:DGVGM} is proved. The case when $U$ is only assumed to be included into a non-coercive subgraph is considered in Section~\ref{sec7}, where the notion of directional $\Omega$-limit set is introduced. Lastly, some extensions, as well as some open questions and conjectures, are presented in Section~\ref{secgeneral}.


\section{A uniform spreading speed result}\label{sec:preliminary}

It is well known since~\cite{KPP} that, for equation~\eqref{homo}, propagation occurs with an {\em asymptotic speed of spreading} equal to~$c^*:=2\sqrt{f'(0)}$, and that the latter coincides with the minimal speed $c$ of {\em traveling fronts}, i.e., solutions of the type
\be\label{eqvp}
u(t,x)=\vp(x\cdot e-ct),\ \ 0=\vp(+\infty)<\vp<\vp(-\infty)=1,\ \ c\in\R,\ \ e\in\Sph,
\ee
where $\vp:\R\to(0,1)$ is of class $C^2(\R)$ and decreasing. The precise result is derived in~\cite{AW} and asserts that, for any solution $u$ to~\eqref{homo} with an  initial condition $0\le u_0\le 1$ which is compactly supported and fulfills $\inf_B u_0>0$ for some ball $B\subset\R^N$ with positive measure, it holds that
\be\label{c<c*}
\forall\,c\in(0,c^*),\quad\inf_{|x|\le ct}u(t,x)\to1\as t\to+\infty,
\ee
\Fi{c>c*}
\forall\,c>c^*,\quad\sup_{|x|\geq ct}u(t,x)\to0\as t\to+\infty,
\Ff
with $c^*=2\sqrt{f'(0)}$.

In the sequel, we will need the following uniform version of the above properties for initial data of the form $u_0=\1_U$ with $U$ unbounded. 

\begin{proposition}\label{pro:uniformspreading}
Let $u$ be a solution of~\eqref{homo} emerging from an initial datum $u_0=\1_U$ with $U\subset\R^N$. Then, for any $\delta>0$ such that $U_\delta\neq\emptyset$, the following convergences hold:
\begin{align}
\label{c<c*uniform}
\forall\,c\in(0,c^*), \quad & \ \displaystyle\inf_{x\in\R^N,\,\dist(x,U_\delta)\leq ct} u(t,x)\to1\as t\to+\infty,\\
\label{c>c*uniform}
\forall\,c>c^*, \quad & \ \displaystyle\sup_{x\in\R^N,\,\dist(x,U)\geq ct}u(t,x)\to0\as t\to+\infty.
\end{align}
\end{proposition}

The reason why~\eqref{c<c*uniform} involves $U_\delta$ instead of $U$ is to neglect subsets of $U$ (such as isolated points) which do not affect the solution $u$ at positive times. The role of hypothesis~\eqref{UUdelta} in our main results is precisely that it allows us to replace $U_\delta$ with~$U$ in~\eqref{c<c*uniform}.

The uniform ``invasion property''~\eqref{c<c*uniform} will be immediately deduced from~\eqref{c<c*}. Instead, property~\eqref{c>c*uniform} does not follow from~\eqref{c>c*}. In order to prove it we construct a family of supersolutions whose upper level sets are given by the exterior of balls retracting with a speed larger, but arbitrarily close, to~$c^*$. These supersolutions will also directly be used to prove \thm{DGgeneral}. Here is their~construction.

\begin{lemma}\label{lem:super}
For any $c>c^*$ and~$\lambda>0$, there exist $R>0$ $($depending on $N,f,c$ and~$\lambda$$)$ and a family of positive functions~$(v^T)_{T>0}$ of class $C^2(\R\times\R^N)$ such that, for each~$T>0$,~$v^T$~is a supersolution to~\eqref{homo} in $\R\times\R^N$ and satisfies
\Fi{vT10}
\left\{\baa{ll}
v^T(0,x)\geq 1, & \text{for all $x$ such that }|x|\geq R+cT,\vspace{3pt}\\
v^T(t,0)<\lambda, & \text{for all }t\in[0,T].\eaa\right.
\Ff
\end{lemma}

\begin{proof}
The functions $v^T$ will be constructed as the sums of finitely many positive solutions to~\eqref{homo}, hence they will be supersolutions to~\eqref{homo} due to the following standard consequence of the Fisher-KPP condition~\eqref{fkpp}:
\Fi{sumsuper}
\forall\,a,b\geq0,\quad f(a+b)\leq f(a)+f(b).
\Ff
To show the above inequality, assume to fix the ideas that $a\leq b$, with $b>0$ (otherwise the inequality trivially holds because $f(0)=0$). Then, observing that the function $f$, which we recall is extended by $0$ outside $[0,1]$, fulfills the third condition in~\eqref{fkpp} on the whole half-line $(0,+\infty)$, we~get
$$f(a+b)\leq \frac{f(b)}{b}(a+b)\leq f(a)+f(b).$$
	
Now, let $c>c^*$ and $\lambda>0$. Take $\e\in(0,1/2)$ small enough to have that $(1-\e)c>c^*$. Consider a finite subset $\mc{S}$ of the unit ball $\mathbb{S}^{N-1}$ such that
$$\forall\,e\in\mathbb{S}^{N-1},\quad\dist(e,\mc{S})\leq\e.$$
The set $\mc{S}$ only depends on $N$ and $\e$, which in turn depends on $f,c$. We define the family of positive functions~$(v^T)_{T>0}$, of class $C^2(\R\times\R^N)$, by
$$v^T(t,x):=2\sum_{e'\in\mc{S}}\vp\big(x\.e'-c^*(t-T)+R/2\big),$$
where $\vp$ is the traveling front with speed $c^*$, as in~\eqref{eqvp}, normalized by $\vp(0)=1/2$, and $R>0$ is such that $\vp(R/2)<\lambda/(2n)$, with $n$ being the number of elements of $\mc{S}$. With~this choice of $R$ (which only depends on $N,f,c,\lambda$) we have, as desired, $v^T(t,0)<\lambda$ for all $T>0$, $t\in[0,T]$. Observe that each term of the sum in the definition of $v^T$ is a positive solution to~\eqref{homo}, hence $v^T$ is a positive supersolution to~\eqref{homo} as discussed at the beginning of the proof.

It remains to check that each $v^T$ fulfills the first property in~\eqref{vT10}. Let $T>0$ and $x\in\R^N$ satisfy $|x|\geq R+cT$, and consider $e'\in\mc{S}$ such that 
$$\left|e'+\frac x{|x|}\right|\le\e.$$
We have that
$$v^T(0,x)\geq2\vp(x\.e'+c^*T+R/2).$$
Since $x\.e'\leq -|x|+\e|x|$, we deduce
\[\begin{split}
x\.e'+c^*T+R/2 &\leq (\e-1)(R+cT)+c^*T+R/2\\
&\leq (\e-1/2)R+[(\e-1)c+c^*]T,
\end{split}\]	
which is negative because $\e<1/2$ and $(1-\e)c>c^*$. It follows that $v^T(0,x)\geq2\vp(0)=1$. This concludes the proof.
\end{proof}

We can now prove the uniform spreading speed result. 

\begin{proof}[Proof of Proposition~$\ref{pro:uniformspreading}$]
We start with~\eqref{c<c*uniform}. Let $\delta>0$ be such that $U_\delta\neq\emptyset$ and let~$v$ be the solution to~\eqref{homo} emerging from the initial datum $v_0=\1_{B_\delta}$. Take $c\in(0,c^*)$. For any $x_0\in U_\delta$, we have that $u_0(x_0+\.)\geq v_0$ in $\R^N$ and therefore $u(t,x_0+x)\geq v(t,x)$ for all $t\geq0$, $x\in\R^N$ thanks to the parabolic comparison principle. Applying~\eqref{c<c*} to $v$ we deduce that
$$1\ge\inf_{x_0\in U_\delta,\, |x|\leq ct} u(t,x_0+x)\geq \inf_{|x|\leq ct}v(t,x)\to1\as t\to+\infty.$$
This is property~\eqref{c<c*uniform}.
	
Let us turn to~\eqref{c>c*uniform}. Take $c>c^*$, $c'\in(c^*,c)$,~$\lambda>0$, and consider the family of supersolutions~$(v^T)_{T>0}$ given by Lemma~\ref{lem:super}, associated with the speed $c'$. Namely, there exists $R>0$ such that they fulfill~\eqref{vT10} with $c'$ instead of $c$. For $T\geq R/(c-c')$ it holds that $cT\geq R+c'T$ and thus, if $x_0\in\R^N$ is such that $\dist(x_0,U)\geq cT$, then $\dist(x_0,U)\geq R+c'T$ and~\eqref{vT10} implies that $0\le u_0(x_0+\cdot)\leq v^{T}(0,\cdot)$ in $\R^N$, hence by comparison $0\le u(T,x_0)\leq v^{T}(T,0)<\lambda$. This shows that
$$\forall\,T\geq \frac{R}{c-c'},\quad 0\le \sup_{x_0\in\R^N,\, \dist(x_0,U)\geq cT}u(T,x_0)\leq \lambda.$$
From this, property~\eqref{c>c*uniform} follows by the arbitrariness of $c>c^*$ and $\lambda>0$.
\end{proof}


\section{Asymptotic one-dimensional symmetry}\label{sec:a1D}

We first prove our most general symmetry result, Theorem~\ref{thm:DGgeneral}, in Section~\ref{sec52}, after the proof of a preliminary approximation result in Section~\ref{sec51}. Next, we derive the proof of Theorem~\ref{thm:DG} in Section~\ref{sec53}, by showing that its hypotheses imply the ones of Theorem~\ref{thm:DGgeneral}. Lastly, Section~\ref{sec54} provides some counterexamples to the main results when the assumptions~\eqref{UUdelta} or~\eqref{ballcone} are not satisfied, while the proof of Theorem~\ref{thm:global} is carried out in Section~\ref{sec55}.


\subsection{An approximation result by truncation of the initial datum}\label{sec51}

The cornerstone of the proof of Theorem~\ref{thm:DGgeneral}, hence of the whole paper, is an approximation result. Before stating it, let us introduce some notation. We recall that $B_R(x_0)$ and $B_R$ stand for the balls in $\R^N$ of radius $R$ and center $x_0$ and $0$ respectively. A generic point~$x\in\R^N$ will sometimes be denoted by $(x',x_N)\in\R^{N-1}\times\R$, and the ball in~$\R^{N-1}$ of radius $R$ and center~$x_0'\in\R^{N-1}$ is denoted by~$B_R'(x_0')$, or just $B_R'$ if $x_0'=0$.

\begin{lemma}\label{lem:approximation}
Let $u$ be a solution to~\eqref{homo} with an initial condition  $u_0=\1_U$, where $U\subset\R^N$ satisfies, for some $\delta>0$, $L>0$, and $\sigma\in(0,c^*/2)$,
\be\label{assapprox}
U_\delta\cap B_L\neq\emptyset \quad\text{ \ and \ }\quad U \setminus (B'_L\times\R)\,\subset\,\Big\{(x',x_N)\in\R^{N-1}\times\R \ : \ x_N\leq\frac{\sigma}{2c^*}|x'|\Big\}.
\ee
Let $(u^R)_{R>0}$ be the solutions to~\eqref{homo} emerging from the initial data $(u_0^R)_{R>0}$ defined by
$$u_0^R=\1_{U\cap(B_R'\times\R)}.$$
Then, for any $\e>0$, there exists $\tau_\e>0$, only depending on $N,f,\delta,L,\sigma,\e$, such~that
\Fi{utaue}
\forall\,\tau\geq\tau_\e,\quad\big\|u(\tau,\.)-u^{3\sigma\tau}(\tau,\.)\big\|_{C^1(B'_{\sigma\tau}\times\R^+)}<\e.
\Ff
\end{lemma}		

\begin{proof}
In order to get the $C^1$ estimate~\eqref{utaue}, it is sufficient to show an $L^\infty$ estimate at a later time, that is, 
\Fi{utau+1}
\forall\,\tau\geq\tau_\e,\quad\big\|u(\tau+1,\.)-u^{3\sigma\tau}(\tau+1,\.)\big\|_{L^\infty(\mc{C}_\tau)}<\e,
\Ff
where $\tau_\e>0$ depends on $N,f,\delta,L,\sigma$ and $\e$, and $\mc{C}_\tau$ is the half-cylinder
$$\mc{C}_{\tau}:={B'_{\sigma\tau}}\times\R^+={B'_{\sigma\tau}}\times(0,+\infty).$$
Indeed, once~\eqref{utau+1} is proved, observing that $u-u^{3\sigma\tau}$ is nonnegative (by the comparison principle) and it solves a parabolic equation that can be written in linear form, one infers from the parabolic Harnack inequality and interior estimates, given for instance by~\cite{Krylov}, that~\eqref{utaue} holds with $\e$ replaced by $C\e$, where $C$ only depends on $f$ and $N$. Then, to prove the lemma it is sufficient to derive~\eqref{utau+1} for an arbitrary $\e>0$.
	
Fix $\e>0$. Consider the family of solutions $(w^R)_{R>0}$ to~\eqref{homo} emerging from the initial data~$(w_0^R)_{R>0}$ given~by
$$w_0^R=\1_{W^R}\quad\text{with}\quad W^R:=\Big\{(x',x_N)\in\R^{N-1}\times\R\ : \ |x'|\geq R,\ x_N\leq\frac{\sigma}{2c^*}|x'|\Big\}.$$
By~\eqref{assapprox} it holds that $u_0\leq\min(u_0^R+w_0^R,1)$ in $\R^N$ for every $R\geq L$. Hence, since the KPP condition~\eqref{fkpp} yields~\eqref{sumsuper}, the minimum between $1$ and the sum of two solutions ranging in $[0,1]$ is a supersolution. We infer by comparison that, for $R\geq L$,
$$0\le u^R\leq u\leq\min(u^R+w^R,1)\ \hbox{ in }[0,+\infty)\times\R^N.$$		
Thus, property~\eqref{utau+1} holds for some $\tau_\e\ge L/(3\sigma)$ if we show that 
\Fi{Ctau}
\forall\,\tau\geq\tau_\e,\quad\sup_{x\in\mc{C}_\tau}\big[\min\big(u^{3\sigma\tau}(\tau+1,x)+w^{3\sigma\tau}(\tau+1,x)\,,\,1\big)-u^{3\sigma\tau}(\tau+1,x)\big]<\e.
\Ff

In order to prove \eqref{Ctau}, we consider a value $c\in(2\sigma,c^*)$, that will be specified later, and divide the half-cylinder~$\mc{C}_\tau$ into the subsets 
$$\mc{C}_\tau^i:=\big(B'_{\sigma\tau}\times\R^+\big)\cap B_{c\tau},\qquad\mc{C}_\tau^e:=\big(B'_{\sigma\tau}\times\R^+\big)\setminus B_{c\tau}.$$ 
	
Let us first deal with the set $\mc{C}_\tau^i$, with arbitrary $c\in(2\sigma,c^*)$. We want to show that $u^{3\sigma\tau}(\tau+1,\.)>1-\e$ there for $\tau$ large. By hypothesis, there exists a ball $B_\delta(x_0)\subset U$ with $|x_0|<L$, hence $B_\delta(x_0)\subset U\cap(B_{L+\delta}'\times\R)$. It follows by comparison that $u^{L+\delta}(t,x)\geq v(t,x-x_0)$ for all $t\ge0$ and $x\in\R^N$, where $v$ is the solution to~\eqref{homo} with initial datum $\1_{B_1}$. Then, applying the spreading property~\eqref{c<c*} to $v$ we infer that, for~$\tau_1>0$ large enough, depending on $N,f,c,\delta,\e$ (recall that $c^*=2\sqrt{f'(0)}$), it holds that
$$\forall\,\tau\geq\tau_1,\quad\inf_{x\in B_{\frac{c^*+c}2\,\tau}(x_0)}u^{L+\delta}(\tau+1,x)>1-\e.$$
Since the family $(u^R)_{R>0}$ is nondecreasing with respect to~$R$, by the parabolic comparison principle, up to increasing $\tau_1$ so that $3\sigma\tau_1\geq L+\delta$, we derive
$$\forall\,\tau\geq\tau_1,\quad\inf_{x\in B_{\frac{c^*+c}2\,\tau}(x_0)}u^{3\sigma\tau}(\tau+1,x)>1-\e$$
(with $\tau_1$ also depending on $L$ and $\sigma$). Finally, from the inclusions $\mc{C}_\tau^i\subset B_{c\tau}\subset B_{c\tau+L}(x_0)$, which hold for all $\tau>0$, and $B_{c\tau+L}(x_0)\subset B_{\frac{c^*+c}2\tau}(x_0)$, which holds if $\frac{c^*-c}2\tau\geq L$, we find a quantity $\tau_2>0$ depending on $N,f,c,\delta,L,\sigma,\e$ such that
$$\forall\,\tau\geq\tau_2,\quad\inf_{x\in\mc{C}_\tau^i}u^{3\sigma\tau}(\tau+1,x)\geq\inf_{x\in B_{\frac{c^*+c}2\tau}(x_0)}u^{3\sigma\tau}(\tau+1,x)>1-\e.$$
This shows that, for any choice of~$c\in(2\sigma,c^*)$, the estimate~\eqref{Ctau} holds when $\mc{C}_\tau$ is replaced by $\mc{C}_\tau^i$ and $\tau_\e$ is equal to the above quantity $\tau_2$.

Let us consider now the set $\mc{C}_\tau^e$. We want to show that $w^{3\sigma\tau}<\e$ there by estimating the distance between $\mc{C}_\tau^e$ and $W^{3\sigma\tau}$ and then applying Lemma~\ref{lem:super}. In this paragraph, $\tau>0$ is arbitrary and $c\in(2\sigma,c^*)$ will be fixed at the end of the paragraph. Take two arbitrary points $x=(x',x_N)\in\mc{C}_\tau^e$ and $y=(y',y_N)\in W^{3\sigma\tau}$. There~holds
$$|x'|<\sigma\tau\leq\frac13|y'|\quad\text{ and }\quad x_N> \sqrt{c^2-\sigma^2}\,\tau,\quad y_N\leq\frac{\sigma}{2c^*}|y'|.$$
We compute
$$|x-y|^2=|x'-y'|^2+(x_N-y_N)^2\geq\frac49|y'|^2+(x_N-y_N)^2.$$
If $\sigma|y'|/(2c^*)\geq\sqrt{c^2-\sigma^2}\,\tau$, neglecting $(x_N-y_N)^2$ in the above inequality we get
$$|x-y|^2\geq\frac{16}9(c^*)^2\Big(\frac{c^2}{\sigma^2}-1\Big)\tau^2\geq\frac{16}3(c^*)^2\tau^2$$
since $c>2\sigma>0$. Instead, in the opposite case $\sigma|y'|/(2c^*)<\sqrt{c^2-\sigma^2}\,\tau$, one has $y_N\leq\sigma|y'|/(2c^*)<\sqrt{c^2-\sigma^2}\,\tau<x_N$, whence 
\begin{align}
|x-y|^2 &\geq\frac49|y'|^2+\Big(\sqrt{c^2-\sigma^2}\,\tau-\frac{\sigma}{2c^*}|y'|\Big)^2\\
&= \frac49|y'|^2+(c^2-\sigma^2)\tau^2+\frac{\sigma^2}{4(c^*)^2}|y'|^2-\frac{\sigma}{c^*}\sqrt{c^2-\sigma^2}\,\tau|y'|,
\end{align}
and we estimate the negative terms by observing that
$$\frac49|y'|^2-\sigma^2\tau^2-\frac{\sigma}{c^*}\sqrt{c^2-\sigma^2}\,\tau|y'|\geq|y'|\Big(\frac13|y'|-\frac{\sigma}{c^*}\sqrt{c^2-\sigma^2}\,\tau\Big)\geq0$$
since $|y'|/3\ge\sigma\tau\ge\sigma\sqrt{c^2-\sigma^2}\tau/c^*$. Thus, in such case one has
$$|x-y|^2\geq c^2\tau^2+\frac{\sigma^2}{4(c^*)^2}|y'|^2\geq c^2\tau^2+\frac{9\sigma^4\tau^2}{4(c^*)^2},$$
which is larger than $(c^*)^2\tau^2$ for $c\in(2\sigma,c^*)$ close enough to $c^*$, depending on $c^*=2\sqrt{f'(0)}$ and $\sigma$ only. Summing up, we have shown the existence of some $c\in(2\sigma,c^*)$ and $c'>c^*$, depending on $c^*=2\sqrt{f'(0)}$ and $\sigma$, such that
\Fi{distW}
\forall\,\tau>0, \ \forall\,x\in\mc{C}_\tau^e,\quad\dist(x,W^{3\sigma\tau})\geq c'\tau.
\Ff
At this point, we invoke the supersolutions $(v^T)_{T>0}$ provided by Lemma~\ref{lem:super}, associated with $\t c:=(c^*+c')/2$ and $\lambda=\e$; they satisfy~\eqref{vT10} with $\t c$ instead of $c$ and a quantity~$R$ depending on~$N,f,\t c,\e$ (hence, $R$ depends on $N,f,\sigma,\e$, since $\t c$ depends on $c^*=2\sqrt{f'(0)}$ and~$c'$ and the latter depends on $c^*,\sigma$). Take $\tau_3>0$ large enough (depending on $R,\t c,c'$, hence on $N,f,\sigma,\e$) so that $R+\t c(T+1)\le c'T$ for all $T\geq\tau_3$. Thus, on one hand, $v^{T+1}(0,x)\geq1$ for~$|x|\geq c'T$ and~$T\geq\tau_3$. On the other hand, for all $\tau>0$ and $x_0\in\mc{C}_\tau^e$, we know from~\eqref{distW} that~$B_{c'\tau}(x_0)\cap W^{3\sigma\tau}=\emptyset$, which implies that $w_0^{3\sigma\tau}(x+x_0)=0$ for~$|x|<c'\tau$. This means that, for $\tau\geq\tau_3$, $w^{3\sigma\tau}(0,\.+x_0)\le v^{\tau+1}(0,\cdot)$ in $\R^N$, and thus $w^{3\sigma\tau}(t,\.+x_0)\le v^{\tau+1}(t,\cdot)$ in~$\R^N$ for all $t\ge0$ by comparison. We conclude by~\eqref{vT10} that
$$\forall\,\tau\geq\tau_3,\ \forall\,x_0\in\mc{C}_\tau^e,\quad w^{3\sigma\tau}(\tau+1,x_0)\leq v^{\tau+1}(\tau+1,0)<\e.$$
This yields that~\eqref{Ctau} holds in the set $\mc{C}_\tau^e$ too, for a suitable choice of $c$ depending on~$c^*=2\sqrt{f'(0)}$ and $\sigma$, and for all $\tau\ge\tau_3>0$ with $\tau_3$ depending on $N,f,\sigma,\e$. Therefore~\eqref{Ctau} holds true in the whole $\mc{C}_\tau$, for some $\tau_\e\ge\max(\tau_2,\tau_3)>0$ depending on~$N,f,\delta,L,\sigma,\e$. The proof of the lemma is complete.
\end{proof}


\subsection{Proof of Theorem~\ref{thm:DGgeneral}}\label{sec52}

Lemma \ref{lem:approximation} allows us to derive Theorem~\ref{thm:DGgeneral} by applying the reflection argument ``\`a~la~Jones''~\cite{J} to the solution with the truncated initial datum. This actually yields an additional information about the direction in which $u$ becomes locally one-dimensional, that will be used to prove one inclusion of the characterization of the set~$\mc{E}$ in \thm{E}. Here is the description of the $\O$-limit set that shows in particular Theorem~\ref{thm:DGgeneral}.

\begin{theorem}\label{thm:psi}
Let $u$ be a solution of~\eqref{homo} with an initial datum $u_0=\1_{U}$ such that $U\subset\R^N$ satisfies~\eqref{UUdelta} and~\eqref{ballcone}. Let $\psi\in\Omega(u)$ and let $\seq{t}$ and $\seq{x}$ be the corresponding sequences given by definition~\eqref{def:Omega}. The following properties hold:
\begin{enumerate}[$(i)$]
\item if $\ \displaystyle\liminf_{n\to+\infty}\frac{\dist(x_n,U)}{t_n}<c^*$ then $\,\psi\equiv1$;
\item if $\ \displaystyle\limsup_{n\to+\infty}\frac{\dist(x_n,U)}{t_n}>c^*$ then $\,\psi\equiv0$;
\item if $\ \displaystyle\lim_{n\to+\infty}\frac{\dist(x_n,U)}{t_n}=c^*$ and if $(\xi)_{n\in\N}$ is any sequence such that $\xi_n\in\pi_{x_n}$ for all $n\in\N$, then, up to extraction of a subsequence, 
$$\frac{x_n-\xi_n}{|x_n-\xi_n|}\to e\in\Sph\ \hbox{ as }n\to+\infty,$$
and $\psi(x)\equiv\Psi(x\.e)$ for some function~$\Psi\in C^2(\R)$ which is either constant or strictly decreasing.
\end{enumerate}
Hence, in any case, $\psi$ is one-dimensional, and it is either constant or strictly monotone.
\end{theorem}

\begin{proof}
Statement $(ii)$ is a direct consequence of property~\eqref{c>c*uniform}, because in such a case there is $c>c^*$ such that, for any $x\in\R^N$, $\dist(x_n+x,U)>c t_n$ for infinitely many~$n\in\N$.

In the case $(i)$, there is $c\in(0,c^*)$ such that $\dist(x_n,U)<c t_n$ for infinitely many~$n\in\N$. Hence, since $\dist(\.,U_\delta)\leq \dist(\.,U)+d_{\mc{H}}(U,U_\delta)$, for given $c'\in(c,c^*)$ it follows from hypothesis~\eqref{UUdelta} that, for any $x\in\R^N$, $\dist(x_n+x,U_\delta)<c' t_n$ for infinitely many~$n\in\N$. The~invasion property~\eqref{c<c*uniform} eventually yields that $\psi\equiv1$.

We are left with statement $(iii)$, that is, calling for short $k_n:=\dist(x_n,U)$, we now~have
\Fi{kntn}
\frac{k_n}{t_n}\to c^*\as n\to+\infty,
\Ff
In~order to show that $\psi$ is one-dimensional, we proceed in several steps: we first define some new convenient coordinate systems; next, assuming by contradiction that $\psi$ is not one-dimensional, we show that a line orthogonal to a level set of $u$ at time $t_n$ is far from a suitable half-cylinder with radius of order $t_n$, which in the following step is used to define an approximation of~$u$ through a truncation of its initial support~$U$. Then Lemma~\ref{lem:approximation} will ensure that the error in this approximation is small (this is where the geometric assumption~\eqref{ballcone} is used), and this in turn will allow us to obtain a contradiction by applying Jones' reflection argument to the solution with the truncated initial support. In the final step of the proof, we will derive the monotonicity of $\psi$ in the desired direction by using again the reflection argument together with Lemma~\ref{lem:approximation}. 
	
\medskip
\noindent{\it{Step 1: coordinates transformations.}} For $n\in\N$, let $\xi_n$ belong to the set $\pi_{x_n}$ of the projections of $x_n$  onto~$\ol U$ (i.e., $\xi_n\in\ol U$ and $|x_n-\xi_n|=k_n$). Up to extraction of a subsequence, we have by~\eqref{kntn} that~$k_n>0$ for all $n\in\N$. We~set
$$e_n:=\frac{x_n-\xi_n}{k_n}=\frac{x_n-\xi_n}{|x_n-\xi_n|}.$$
Next, we consider a family of $N\times N$ orthogonal transformations~$(M_n)_{n\in\N}$ such that $M_n(\mathrm{e}_N)=e_n$, with $\mathrm{e}_N:=(0,\cdots,0,1)$. Up to subsequences, $\seq{e}$ and $(M_n)_{n\in\N}$ converge respectively to some direction $e\in\Sph$ and some orthogonal transformation $M$, with $M(\mathrm{e}_N)=e$. We define, for $t\geq0$ and $x\in\R^N$, 
$$u_n(t,x):=u(t,\xi_n+M_n(x)).$$
These are still solutions to~\eqref{homo}, because the equation is invariant under isometry. It follows that 
\Fi{untopsi}
u_n(t_n,k_n\mathrm{e}_N+x)=u(t_n,x_n+M_n(x))\longrightarrow\psi(M(x))=:\t\psi(x)\as n\to+\infty,
\Ff
locally uniformly in $x\in\R^N$.	Moreover, $u_n(0,\.)=\1_{U_n}$ with 
$$U_n:=M_n^{-1}(U)-\{M_n^{-1}(\xi_n)\}$$
The set $U_n$ is a rigid transformation of $U$ and it is constructed in a way that $0\in\ol{U_n}$ is an orthogonal projection of $k_n\mathrm{e}_N$ onto $\ol{U_n}$, whence $\dist(k_n\mathrm{e}_N,\overline{U_n})=k_n\to+\infty$ as $n\to+\infty$ by~\eqref{kntn} and therefore the geometric assumption~\eqref{ballcone} (which, being invariant by isometries, is fulfilled by $U_n$ too) yields
\Fi{Unsubset}
U_n\subset\big\{(x',x_N)\in\R^{N-1}\times\R \ : \ x_N\leq\alpha_n|x'|\big\}\quad\text{with $\alpha_n\to0$ as $n\to+\infty$}.
\Ff 
	
\medskip
\noindent{\it{Step 2: the choice of the truncation.}} Firstly, by~interior parabolic estimates, the~$L^\infty_{loc}(\R^N)$ convergence~\eqref{untopsi} holds true in $C^2_{loc}(\R^N)$. We claim~that
\be\label{claimpsi}
\nabla_{\!x'}\t\psi\equiv0\ \hbox{ in $\R^N$},
\ee
which will yield, in the original coordinate system, 
$$\partial_{e'}\psi\equiv\nabla\t\psi\.(M^{-1}(e'))\equiv0\ \hbox{ in }\R^N$$
for any direction $e'\in\Sph$ such that $M^{-1}(e')\perp\mathrm{e}_N$, that is, $e'\perp M(\mathrm{e}_N)=e$. This will imply that $\psi(x)\equiv\Psi(x\.e)$ for some $\Psi\in C^2(\R)$.

Assume by contradiction that the above claim~\eqref{claimpsi} fails, that is, that $\nabla_{\!x'}\t\psi(\bar x)\neq0$ for some~$\bar x\in\R^N$. Let us call for short $\beta:=\nabla\tilde\psi(\bar x)$, hence $\beta=(\beta',\beta_N)\in\R^{N-1}\times\R$ with~$\beta'\neq0$, and it holds that
\Fi{DuntoDpsi}
\nabla u_n(t_n,k_n\mathrm{e}_N+\bar x)\to\beta\as n\to+\infty.
\Ff
	
Take a real number $\vt>0$, that will be fixed at the end of this paragraph. Let~$(H_n)_{n\in\N}$ be the family of closed half-cylinders in $\R^N$ defined~by
$$H_n:=\ol{B'_{\vt t_n}}\times(-\infty,\vt t_n].$$
Consider also the conical sets $(V_n)_{n\in\N}$ given by
\be\label{defVn}
V_n:=\big\{k_n\mathrm{e}_N+\bar x+s(\beta+\zeta):s\in\R,\,\zeta\in B_\vt\big\}.
\ee
We look for $\vt>0$ small enough so that 
\Fi{HnVn}
H_n\cap V_n=\emptyset\quad \text{for all $n$ sufficiently large}.
\Ff
To this end, consider a generic point $P=(P',P_N)\in V_n$. It can be written~as $P=k_n\mathrm{e}_N+\bar x+s(\beta+\zeta)$~for some $s\in\R$, $\zeta\in B_\vt$. Suppose that $P_N\leq\vt t_n$,~thus
$$\vt t_n\geq k_n-|\bar x|-(|\beta_N|+\vt) |s|.$$
By~\eqref{kntn}, for $n$ large enough we have $k_n>(c^*/2)t_n$ and therefore
\Fi{|s|}
|s|\geq \frac{(c^*/2-\vt)t_n-|\bar x|}{|\beta_N|+\vt}.
\Ff
The component $P'$ is estimated as follows:
$$|P'|\geq (|\beta'|-\vt)|s|-|\bar x|.$$
We impose $\vt<|\beta'|$ and, for $n$ sufficiently large, we can invoke~\eqref{|s|} and get
$$|P'|-\vt t_n\geq\left(\frac{|\beta'|-\vt}{|\beta_N|+\vt}(c^*/2-\vt)-\vt\right)t_n-\left(\frac{|\beta'|}{|\beta_N|}+1\right)|\bar x|.$$
Thus, one can find~$\vt\in(0,\min(|\beta'|,c^*/2))$ sufficiently small, only depending on $\beta'$, $\beta_N$ and~$c^*$, such that, for $n$ large, $|P'|>\vt t_n$, i.e.~$P\notin H_n$. This means that, with this choice of~$\vt$, condition~\eqref{HnVn} holds. This fixes our choice of $\vt$ and thus of the family of half-cylinders~$(H_n)_{n\in\N}$. 

\medskip
\noindent{\it{Step 3: the approximation procedure.}} We now apply Lemma \ref{lem:approximation} to the sequence of solutions~$\seq{u}$. Take $\delta>0$ from hypothesis~\eqref{UUdelta}, and call
\be\label{defL2}
\sigma:=\frac\vt3>0,\qquad L:=d_{\mc{H}}(U,U_\delta)+1>0,\qquad\e:=\frac\vt2>0,
\ee
where $\vt>0$ is given in the previous step. One has $0<\sigma<\vt<c^*/2$ and $0<L<+\infty$ by~\eqref{UUdelta}. We further have, on the one hand, that
\be\label{Undelta}
(U_n)_\delta\cap B_L\neq\emptyset,
\ee
because $0\in\ol U_n$ and $d_{\mc{H}}(U_n,(U_n)_\delta)=d_{\mc{H}}(U,U_\delta)<L$. On the other hand, it follows from~\eqref{Unsubset} that, for~$n$~large,  
$$U_n \subset\,\Big\{(x',x_N)\in\R^{N-1}\times\R \ : \ x_N\leq\frac{\sigma}{2c^*}|x'|\Big\}.$$
This means that the sets $U_n$ fulfill the hypotheses~\eqref{assapprox} of Lemma \ref{lem:approximation} for $n$ large enough. Therefore, for such values of $n$, considering the solution $u_n^{\vt t_n}$ of~\eqref{homo} whose initial datum is given by the indicator function of the set
$$U_n\cap(B_{\vt t_n}'\times\R),$$
the estimate~\eqref{utaue} implies 
\be\label{unthetan}
\big\|u_n(t_n,\.)-u_n^{\vt t_n}(t_n,\.)\big\|_{C^1(B'_{\vt t_n/3}\times\R^+)}<\frac\vt2,
\ee
provided that $t_n$ is larger than a quantity depending on $\vt$ but not on~$n$. This means that the above estimate holds for all $n$ sufficiently large. Furthermore,~\eqref{Undelta} implies that, for all $n$ large enough,
\be\label{Undelta2}
U_n\cap(B'_{\vt t_n}\times\R)\ \supset\ B_\delta(y_n),
\ee
for some $y_n\in B_L$. Since $u_n$ fulfills~\eqref{DuntoDpsi}, one then infers from~\eqref{unthetan} that
\Fi{untheta2}
\big|\nabla u_n^{\vt t_n}(t_n,k_n\mathrm{e}_N+\bar x)-\beta\big|<\vt\quad\text{for all $n$ sufficiently large}.
\Ff
This entails that, for such values of $n$, the line $\Gamma_n$ passing through the point $k_n\mathrm{e}_N+\bar x$ and directed as $\nabla u_n^{\vt t_n}(t_n,k_n\mathrm{e}_N+\bar x)$ is contained in the set $V_n$ defined in~\eqref{defVn}, and therefore, by~\eqref{HnVn},
$$\big(\Gamma_n\cap H_n\big)\subset \big(V_n \cap H_n\big)=\emptyset\quad \text{for all $n$ sufficiently large}.$$

Next, owing to~\eqref{Unsubset}, we also have that $\overline{U_n}\cap(\overline{B_{\vt t_n}'}\times\R)\subset H_n$ for all $n$ sufficiently large, hence
\be\label{suppu0n1}
\supp u_n^{\vt t_n}(0,\.)\subset H_n\quad \text{for all $n$ sufficiently large},
\ee
where $\supp u_n^{\vt t_n}(0,\.)$ denotes the support of $u_n^{\vt t_n}$ (one has $u_n^{\vt t_n}(0,\.)\equiv 0$ in $\R^N\setminus H_n$).

\medskip
\noindent{\it{Step 4: the reflection argument.}} Let $H_n,V_n,\Gamma_n$ and $u_n^{\vt t_n}$ be as in the previous steps. For~$n$ large enough, the half-cylinder $H_n$ and the line $\Gamma_n$ are convex, closed and disjoint; we can then separate them with an hyperplane, which, up to translation, can be assumed without loss of generality to contain $\Gamma_n$. Namely, for $n$ large, there exists an open half-space $\O_n$ such that
\be\label{inclusions}
\Gamma_n\subset\partial\O_n\quad\text{and}\quad H_n\subset\O_n.
\ee
Let~$\mc{R}_n$ denote the affine orthogonal reflection with respect to $\partial\O_n$. Then define the function $v^n$ in $[0,+\infty)\times\overline{\O_n}$ by
$$v^n(t,x):=u_n^{\vt t_n}(t,\mc{R}_n(x)).$$
The function $v^n$ coincides with $u_n^{\vt t_n}$ on $[0,+\infty)\times\partial\O_n$. Furthermore $v_n(0,\cdot)$ vanishes identically in $\O_n$, while $u_n^{\vt t_n}=1$ in $B_\delta(y_n)\subset U_n\cap(B'_{\vt t_n}\times\R)\subset H_n\subset\Omega_n$, provided $n$ is large enough for~\eqref{Undelta2},~\eqref{suppu0n1} and~\eqref{inclusions} to hold. Then, for such values of~$n$, it follows from the comparison principle that $v^n\leq u_n^{\vt t_n}$ in~$(0,+\infty)\times\O_n$, and moreover, by the Hopf lemma, that $\partial_{\nu_n} v^n>\partial_{\nu_n} u_n^{\vt t_n}$ on $(0,+\infty)\times\partial\O_n$, where~${\nu_n}$ is the exterior normal to $\O_n$. Since clearly $\partial_{\nu_n} v^n=-\partial_{\nu_n} u_n^{\vt t_n}$ on $(0,+\infty)\times\partial\Omega_n$, this means that~$\partial_{\nu_n} u_n^{\vt t_n}<0$ on $(0,+\infty)\times\partial\O_n$, and thus in particular that $\partial_{\nu_n} u_n^{\vt t_n}(t_n,k_n\mathrm{e}_N+\bar x)<0$, because~$k_n\mathrm{e}_N+\bar x\in\Gamma_n\subset\partial\O_n$. This is however impossible because $\nabla u_n^{\vt t_n}(t_n,k_n\mathrm{e}_N+\bar x)$ is parallel to $\Gamma_n$ and thus orthogonal to~$\nu_n$. We~have reached a contradiction. This therefore shows~\eqref{claimpsi}, and then $\psi(x)\equiv\Psi(x\cdot e)$ for some function $\Psi\in C^2(\R)$.

\medskip
\noindent{\it{Step 5: the large monotonicity property.}} Let us now show that $\Psi$ is nonincreasing, that is, $\partial_{x_N}\t\psi\leq0$ in $\R^N$. We fix an arbitrary $\e>0$ and an arbitrary $\hat x\in\R^N$. As in Step~3, we apply Lemma \ref{lem:approximation} to the sequence of solutions~$\seq{u}$. Namely, we take $\delta>0$ from hypothesis~\eqref{UUdelta} and call
\be\label{sigmaL2}
\sigma:=\frac{c^*}6,\qquad L:=d_{\mc{H}}(U,U_\delta)+1.
\ee
As seen before, with these values, we have that the sets $U_n$ fulfill~\eqref{assapprox} for $n$ large enough. Therefore, calling $u_n^{R}$ the solution with initial datum given by the indicator function of the~set
$$U_n\cap(B_{R}'\times\R),$$
for such large values of $n$, one can apply Lemma \ref{lem:approximation} and infer that
$$\forall\,\tau\geq\tau_\e,\quad\big\|u_n(\tau,\.)-u_n^{c^*\tau/2}(\tau,\.)\big\|_{C^1(B'_{c^*\tau/6}\times\R^+)}<\e,$$
where $\tau_\e$ only depends on $N,f,\delta,L,\e$. As a consequence, for $n$ large enough such that the above holds, and in addition $t_n>\tau_\e$, we deduce that
\Fi{un<e}
\big\|u_n(t_n,\.)-u_n^{c^*t_n/2}(t_n,\.)\big\|_{C^1(B'_{c^*t_n/6}\times\R^+)}<\e\quad\text{for all $n$ sufficiently large}.
\Ff

Next, consider the half-space
$$\O_n:=\R^{N-1}\times(-\infty,k_n+\hat x\.\eN],$$
so that $k_n\eN+\hat x\in\partial\O_n$. By~\eqref{kntn} and~\eqref{Unsubset} it holds that $U_n\cap(B_{c^*t_n/2}'\times\R)\subset\O_n$ for $n$ sufficiently large (also depending on $\hat x$), that is, the support of the initial datum of~$u_n^{c^*t_n/2}$ is contained in $\O_n$. The same reflection argument as in Step~4 eventually yields
$$\partial_{x_N} u_n^{c^*t_n/2}(t_n,k_n\eN+\hat x)\leq0\quad\text{for all $n$ sufficiently large}.$$
Combining this with~\eqref{un<e} we infer that $\partial_{x_N}\tilde\psi(\hat x)\leq\e$, and thus $\partial_{x_N}\tilde\psi\leq0$ in $\R^N$ by the arbitrariness of $\e>0$ and $\hat x\in\R^N$.

\medskip
\noindent{\it{Step 6: the strict monotonicity property.}} We have shown in the above steps that $\psi(x)\equiv\Psi(x\cdot e)$ with $\Psi$ nonincreasing. It remains to show that $\Psi$ is either constant or strictly decreasing in the whole $\R$. To do so, consider the functions $(t,x)\mapsto u(t_n+t,x_n+M_n(x))$ for $n\in\N$, which satisfy the same equation~\eqref{homo} as $u$, but for $t>-t_n$. From standard parabolic estimates, there is a solution $u_\infty:\R\times\R^N\to[0,1]$ of~\eqref{homo}, which is a solution also for~$t\leq0$, and such that, up to extraction of a subsequence,
$$u(t_n+t,x_n+M_n(x))\to u_\infty(t,x)\ \hbox{ as $n\to+\infty$ locally uniformly in $\R\times\R^N$}.$$
We apply the results derived in the Steps~1-5 to the sequences $(t_n-1)_{n\in\N}$ and $(x_n)_{n\in\N}$. Observe that also these sequences fulfill the condition of statement $(iii)$. Moreover, the direction $e$ associated with these new sequences, as defined at the beginning of Step~1, only depends on $(x_n)_{n\in\N}$ and $(\xi_n)_{n\in\N}$, hence it is the same before. We deduce the existence of a nonincreasing function $\Phi:\R\to[0,1]$ such that $u_\infty(-1,x)\equiv\Phi(x\cdot e)$ in $\R^N$. Notice also that $u_\infty(0,x)\equiv\Psi(x\cdot e)$ in $\R^N$. Now, if the function $\Phi$ is constant in $\R$, which means that $u_\infty(-1,\cdot)$ is constant in $\R^N$, then so is $u_\infty(0,\cdot)$ in $\R^N$ (because $u_\infty$ solves~\eqref{homo} for all $t\in\R$), that is, $\Psi$ is constant in $\R$. On the other hand, if $\Phi$ is not constant in $\R$, then, for each $h>0$,
$$u_\infty(-1,x+he)\le u_\infty(-1,x)\ \hbox{ and }\ u_\infty(-1,x+he)\not\equiv u_\infty(-1,x)\ \hbox{ in $\R^N$},$$
hence $u_\infty(0,x+he)<u_\infty(0,x)$ in $\R^N$ from the strong parabolic maximum principle, yielding $\Psi(s+h)<\Psi(s)$ for all $s\in\R$. As a conclusion, $\Psi$ is either constant or strictly decreasing in $\R$. The proof of the theorem is complete.
\end{proof}

Theorem~\ref{thm:DGgeneral} is a direct consequence of Theorem~\ref{thm:psi}. We also point out that in the case $(iii)$ of Theorem~\ref{thm:psi} it may still happen that $\psi$ is constant, and this actually occurs for instance when $U$ is bounded. Indeed, in such a case, $\Psi$ coincides with some translation of the profile $\vp$ of the critical front (hence it is not constant) if and only~if the quantity
$$\dist(x_n,U)-c^*t_n+\frac{N+2}{c^*}\ln t_n$$
stays bounded as $n\to+\infty$ (that is, if and only if $|x_n|-c^*t_n+\frac{N+2}{c^*}\ln t_n$ is bounded as $n\to+\infty$), otherwise $\Psi\equiv1$ if it diverges to $-\infty$ and $\Psi\equiv0$ if it diverges to $+\infty$, see~\cite{D,RRR}.


\subsection{Proof of Theorem~\ref{thm:DG}}\label{sec53}

Let us turn to Theorem~\ref{thm:DG} that, as we now show, is a special case of Theorem~\ref{thm:DGgeneral}. For~this, we need to check that the geometric condition~\eqref{ballcone} is invariant among sets having finite Hausdorff distance from one another. This is done in the following geometric~lemma.

\begin{lemma}\label{lem:ballcone}
For any $U\subset\R^N$, consider the function $\mc{O}(x)$ defined in~\eqref{opening}. Then the~map
$$R\mapsto\sup_{x\in\R^N,\,\dist(x,U)=R}\,\mc{O}(x)$$
is nonincreasing in $(0,+\infty)$. Moreover, for any $U'\subset\R^N$ satisfying $d_{\mc H}(U,U')<+\infty$, then~ $U$ fulfills~\eqref{ballcone} if and only if~$U'$ does $($with the corresponding $\mc{O}$ defined as in~\eqref{opening} with $U'$ instead of $U)$.
\end{lemma}	

\begin{proof}
The monotonicity property involving $\mc{O}$ is readily derived. Consider indeed any
$$0<R'<R.$$
If the set $\{x\in\R^N:\dist(x,U)=R\}$ is empty, then $\sup_{x\in\R^N,\,\dist(x,U)=R}\mc{O}(x)=-\infty$ and the inequality $\sup_{x\in\R^N,\,\dist(x,U)=R}\mc{O}(x)\le\sup_{x\in\R^N,\,\dist(x,U)=R'}\mc{O}(x)$ is trivially true. Assume now that the set $\{x\in\R^N:\dist(x,U)=R\}$ is not empty, and consider any $x$ in this set and any $\xi\in\pi_x$, that is, $\xi\in\overline{U}$ and $|x-\xi|=\dist(x,U)=R$. If $U=\{\xi\}$ then $\mc{O}(x)=-\infty$ by our convention, hence this case is trivial too.

Assume then that $U\neq\{\xi\}$. Consider the point $x':=\xi+(R'/R)(x-\xi)$. Its unique projection onto $\overline{U}$ is $\xi$, that is, $\pi_{x'}=\{\xi\}$. Furthermore, $\dist(x',U)=|x'-\xi|=R'$. One also observes that, for any $y\in U\!\setminus\!\{\xi\}$,
$$\frac{x-\xi}{|x-\xi|}\.\frac{y-\xi}{|y-\xi|}=\frac{x'-\xi}{|x'-\xi|}\.\frac{y-\xi}{|y-\xi|}\le\mc{O}(x')\le\sup_{z\in\R^N,\,\dist(z,U)=R'}\mc{O}(z).$$
Since $x$ with $\dist(x,U)=R$, together with $\xi\in\pi_x$ and $y\in U\!\setminus\!\{\xi\}$, were arbitrary, this shows that
$$\sup_{z\in\R^N,\,\dist(z,U)=R}\mc{O}(z)\leq\sup_{z\in\R^N,\,\dist(z,U)=R'}\mc{O}(z).$$
	
Let us turn to the second statement of the lemma. One considers any two subsets~$U$ and~$U'$ of $\R^N$ satisfying~$d_{\mc{H}}(U,U')<+\infty$. Denote $\pi'_x$ and $\mc{O}'(x)$ the objects defined as in~\eqref{defpix}-\eqref{opening} with $U'$ instead of $U$.

Assume by way of contradiction that $U$ fulfills~\eqref{ballcone} and $U'$ does not. Then there are~$\e>0$ and a sequence $(x_n)_{n\in\N}$ in $\R^N\!\setminus\!\overline{U'}$ such that
\be\label{xnR'n}
0<R'_n:=\dist(x_n,U')\to+\infty\hbox{ as $n\to+\infty,\ $ and $\ \mc{O}'(x_n)\ge2\e>0$ for all $n\in\N$}.
\ee
Calling $d:=d_{\mc{H}}(U,U')<+\infty$, one then has $R_n:=\dist(x_n,U)\to+\infty$ as $n\to+\infty$, and moreover
\be\label{RnR'n}
R_n-d\le R'_n\le R_n+d\ \hbox{ for all }n\in\N.
\ee
Without loss of generality, one has $R_n>0$ for every $n\in\N$. Since $U$ is assumed to satisfy~\eqref{ballcone}, there holds $\limsup_{n\to+\infty}\mc{O}(x_n)\le0$, that is,
\be\label{Oxn}
\mc{O}(x_n)^+:=\max\big(\mc{O}(x_n),0\big)\to0\as n\to+\infty.
\ee
Now, from~\eqref{xnR'n}, for each $n\in\N$, there are $\xi'_n\in\pi'_{x_n}$, that is, $\xi'_n\in\overline{U'}$ and $|x_n-\xi'_n|=\dist(x_n,U')=R'_n>0$, and $y'_n\in U'\!\setminus\!\{\xi'_n\}$ such that
\be\label{xi'ny'n}
\frac{x_n-\xi'_n}{|x_n-\xi'_n|}\cdot\frac{y'_n-\xi'_n}{|y'_n-\xi'_n|}\ge\e>0.
\ee
For each $n\in\N$, consider any $\xi_n\in\pi_{x_n}$, that is, $\xi_n\in\overline{U}$ and $|x_n-\xi_n|=\dist(x_n,U)=R_n>0$, and then there is a point $y_n\in U$ such that
\be\label{d+1}
|y_n-y'_n|\le d_{\mc{H}}(U,U')+1=d+1.
\ee
We estimate from above the quantities in~\eqref{xi'ny'n} by writing:
\be\label{ineqcontra}
\frac{x_n\!-\!\xi'_n}{|x_n\!-\!\xi'_n|}\.\frac{y'_n\!-\!\xi'_n}{|y'_n\!-\!\xi'_n|}\!\leq\!\underbrace{\Big|\frac{x_n\!-\!\xi'_n}{|x_n\!-\!\xi'_n|}\!-\!\frac{x_n\!-\!\xi_n}{|x_n\!-\!\xi_n|}\Big|}_{=:I_{1,n}}\!+\!\underbrace{\Big|\frac{y'_n\!-\!\xi'_n}{|y'_n\!-\!\xi'_n|}\!-\!\frac{y_n\!-\!\xi_n}{|y_n\!-\!\xi_n|}\Big|}_{=:I_{2,n}}\!+\!\underbrace{\frac{x_n\!-\!\xi_n}{|x_n\!-\!\xi_n|}\!\.\!\frac{y_n\!-\!\xi_n}{|y_n\!-\!\xi_n|}}_{=:I_{3,n}}\!.
\ee
This inequality is understood to hold whenever $y_n\neq\xi_n$, which we will show to occur for~$n$ sufficiently large. We will then prove that $I_{1,n}$, $I_{2,n}$, $I_{3,n}$ $\to0$ as $n\to+\infty$, which will eventually contradict~\eqref{xi'ny'n}. In order to estimate $I_{1,n}$, we take~$z_n\in\overline{U}$ such that~$|z_n-\xi'_n|\leq d$ and we compute
\be\label{xxi}\baa{rcl}
(x_n-\xi_n)\.(x_n-\xi'_n) & = & R_n^2+(x_n-\xi_n)\.(\xi_n-z_n)+(x_n-\xi_n)\.(z_n-\xi'_n)\vspace{3pt}\\
& \geq & R_n^2-\mc{O}(x_n)R_n|z_n-\xi_n|-R_nd\vspace{3pt}\\
& \geq & R_n\big(R_n-2(R_n+d)\mc{O}(x_n)^+-d\big),\eaa
\ee
where the last inequality follows from
$$|z_n-\xi_n|\le|z_n-\xi'_n|+|\xi'_n-x_n|+|x_n-\xi_n|\le d+R'_n+R_n\le2(R_n+d).$$
One then derives from~\eqref{RnR'n} and~\eqref{xxi} that
\Fi{I1}
0\le I_{1,n}\leq\sqrt{2-\frac{2\big(R_n-2(R_n+d)\mc{O}(x_n)^+-d\big)}{R'_n}}\leq2\sqrt{\frac{(R_n+d)\mc{O}(x_n)^++d}{R'_n}}.
\Ff
Together with~\eqref{xnR'n}-\eqref{Oxn}, one gets that
\be\label{I1n}
I_{1,n}\to0\as n\to+\infty.
\ee
Next, let us check that $y_n\neq\xi_n$ for $n$ large. We first control $|y'_n-\xi'_n|$ from below. We write
$$|y'_n-x_n|^2=|y'_n-\xi'_n|^2+(R'_n)^2-2(y'_n-\xi'_n)\.(x_n-\xi'_n),$$	
which together with~\eqref{xi'ny'n} and the inequality $|y'_n-x_n|\ge\dist(x_n,U')=R'_n$ yields
$$|y'_n-\xi'_n|\big(|y'_n-\xi'_n|-2\e R'_n\big)\geq|y'_n-x_n|^2-(R'_n)^2\ge 0.$$
Since $y'_n\neq\xi'_n$, this means that
\be\label{y'nxi'n}
|y'_n-\xi'_n|\geq 2\e R'_n.
\ee
Now, using \eqref{xxi} and $R'_n\leq R_n+d$, one infers
\be\label{xixi'nRn}
|\xi_n-\xi'_n|^2=R_n^2+(R'_n)^2-2(x_n-\xi_n,x_n-\xi'_n)\leq 4R_nd+d^2+4R_n(R_n+d)\mc{O}(x_n)^+.
\ee
Gathering together the inequalities~\eqref{d+1},~\eqref{y'nxi'n} and~\eqref{xixi'nRn} shows that
\be\label{ynxn}\baa{rcl}
|y_n-\xi_n| & \geq & |y'_n-\xi'_n|-|y_n-y'_n|-|\xi'_n-\xi_n|\vspace{3pt}\\
& \geq & 2\e R'_n-(d+1)-\sqrt{4R_nd+d^2+4R_n(R_n+d)\mc{O}(x_n)^+}.\eaa
\ee
The right-hand side is positive for all $n$ large enough and is equivalent to $2\e R'_n$ as $n\to+\infty$, because of~\eqref{xnR'n}-\eqref{Oxn}. This means that $y_n\neq\xi_n$ for $n$ large enough. Let us estimate $I_{2,n}$. One has, for $n$ large,
\be\label{I2n}\baa{rcl}
0\le I_{2,n} & \!\!\!=\!\!\! & \displaystyle\sqrt{2-2\frac{(y'_n-\xi'_n)\.(y_n-\xi_n)}{|y'_n-\xi'_n|\times|y_n-\xi_n|}}\vspace{3pt}\\
& \!\!\!=\!\!\! & \displaystyle\sqrt{\frac{|(y'_n-\xi'_n)-(y_n-\xi_n)|^2-(|y'_n-\xi'_n|-|y_n-\xi_n|)^2}{|y'_n-\xi'_n|\times|y_n-\xi_n|}}\vspace{3pt}\\
& \!\!\!\leq\!\!\! & \displaystyle\frac{|(y'_n\!-\!\xi'_n)-(y_n\!-\!\xi_n)|}{\sqrt{|y'_n\!-\!\xi'_n|\!\times\!|y_n\!-\!\xi_n|}}\le\frac{|y_n\!-\!y'_n|+|\xi_n\!-\!\xi'_n|}{\sqrt{|y'_n\!-\!\xi'_n|\!\times\!|y_n\!-\!\xi_n|}}\le\frac{d\!+\!1\!+\!|\xi_n\!-\!\xi'_n|}{\sqrt{|y'_n\!-\!\xi'_n|\!\times\!|y_n\!-\!\xi_n|}},\eaa
\ee
where the last inequality follows from~\eqref{d+1}. Putting together~\eqref{y'nxi'n}-\eqref{I2n} leads to
$$0\le I_{2,n}\leq\frac{d+1+\sqrt{4R_nd+d^2+4R_n(R_n+d)\mc{O}(x_n)^+}}{\sqrt{2\e R'_n}\times\sqrt{2\e R'_n-(d\!+\!1)-\sqrt{4R_nd\!+\!d^2\!+\!4R_n(R_n\!+\!d)\mc{O}(x_n)^+}}}$$
for all $n$ large enough. Using again~\eqref{xnR'n}-\eqref{Oxn}, it follows that 
\be\label{I2nto0}
I_{2,n}\to0\as n\to+\infty.
\ee 
Finally, one has that $0\le I_{3,n}\leq\mc{O}(x_n)\le\mc{O}(x_n)^+$ for all $n$, hence $I_{3,n}\to0$ as $n\to+\infty$, by~\eqref{Oxn}. Together with~\eqref{ineqcontra},~\eqref{I1n} and~\eqref{I2nto0}, one gets that
$$\limsup_{n\to+\infty}\frac{x_n-\xi'_n}{|x_n-\xi'_n|}\.\frac{y'_n-\xi'_n}{|y'_n-\xi'_n|}\le0,$$
a contradiction with~\eqref{xi'ny'n}. The conclusion of the lemma then follows by changing the roles of $U$ and $U'$.
\end{proof}

\begin{proof}[Proof of Theorem~$\ref{thm:DG}$]
If the set $U$ is convex, then the quantity $\mc{O}(x)$ defined by~\eqref{opening} satisfies $\mc{O}(x)\leq0$ for all $x\notin\ol U$, hence condition~\eqref{ballcone} is immediately true in this case. Condition~\eqref{ballcone} holds true as well when $U$ is at bounded Hausdorff distance from a convex set $U'$, thanks to Lemma~\ref{lem:ballcone}.  The conclusion then follows from Theorem~\ref{thm:DGgeneral}.		
\end{proof} 


\subsection{Counterexamples without the conditions~\eqref{UUdelta} or~\eqref{ballcone}}\label{sec54}

This section essentially consists of two propositions, which assert that the conclusions of the main results do not hold in general without the assumptions~\eqref{UUdelta} or~\eqref{ballcone}.

\begin{proposition}\label{proV}
Let $u$ be the solution of~\eqref{homo} with an initial datum $u_0=\1_U$, where~$U$ is the union of two half-spaces with non-parallel boundaries. The set $U$ satisfies~\eqref{UUdelta} but not~\eqref{ballcone}, and $\Omega(u)$ contains some elements which are not one-dimensional.
\end{proposition}

\begin{proof}
We consider the dimension $N=2$ only, since the general case $N\ge3$ follows by trivially extending the functions in the variables $x_3,\cdots,x_N$. Since the equation~\eqref{homo} is invariant by rigid transformations, one can assume without loss of generality that
$$U=\big\{(x_1,x_2)\in\R^2:x_2\le\beta\,|x_1|\big\},$$
for some $\beta>0$. Denote $\alpha:=\arctan\beta\in(0,\pi/2)$. Notice immediately that $U$ satisfies~\eqref{UUdelta}, but it does not satisfy~\eqref{ballcone}. Let $u$ be the solution of~\eqref{homo} in dimension $N=2$, with initial condition $u_0:=\1_U$. Let $v$ be the solution of~\eqref{homo} in dimension $N=1$, with Heaviside initial condition $v_0:=\1_{(-\infty,0]}$. As follows from~\cite{B,HNRR,KPP,L,U1}, there is a function $t\mapsto\zeta(t)$ such that
\be\label{vvarphi}
v(t,x)-\varphi(x-\zeta(t))\to0\ \hbox{ as $t\to+\infty$ uniformly in $x\in\R$},
\ee
where $\varphi$ is the (decreasing) profile of the traveling front $\varphi(x-c^*t)$ solving~\eqref{eqvp}, with  $N=1$, $e=1$, and minimal speed $c=c^*=2\sqrt{f'(0)}$ (furthermore, it is known that $\zeta(t)=c^*t-(3/c^*)\ln t+x_\infty+o(1)$ as $t\to+\infty$, for some real number $x_\infty$). Since~\eqref{homo} is invariant by rigid transformations and since $f(a+b)\le f(a)+f(b)$ for all $a,b\ge0$ by~\eqref{sumsuper}, 
it follows from the definition of~$U$ and the maximum principle that 
\be\label{ineqV}
\begin{array}{l}
\max\big(v(t,x_2\cos\alpha-x_1\sin\alpha),v(t,x_2\cos\alpha+x_1\sin\alpha)\big)\vspace{3pt}\\
\qquad\qquad\qquad\le u(t,x_1,x_2)\le v(t,x_2\cos\alpha-x_1\sin\alpha)+v(t,x_2\cos\alpha+x_1\sin\alpha)\end{array}
\ee
for all $t\ge0$ and $(x_1,x_2)\in\R^2$. Together with~\eqref{vvarphi}, one gets that
$$\liminf_{t\to+\infty}u\Big(t,0,\frac{\zeta(t)+\varphi^{-1}(1/2)}{\cos\alpha}\Big)\ge\frac12\quad 
\hbox{and}\quad\limsup_{t\to+\infty}u\Big(t,0,\frac{\zeta(t)+\varphi^{-1}(1/8)}{\cos\alpha}\Big)\le\frac14,$$
where $\varphi^{-1}:(0,1)\to\R$ denotes the reciprocal of the decreasing function $\varphi$. Consider now any sequence $(t_n)_{n\in\N}$ diverging to $+\infty$. Up to extraction of a subsequence, the functions $(x_1,x_2)\mapsto u(t_n,x_1,\zeta(t_n)/\cos\alpha+x_2)$ converge in $C^2_{loc}(\R^2)$ to a function $\psi:\R^2\to[0,1]$, which then belongs to $\Omega(u)$. One has
$$\psi\Big(0,\frac{\varphi^{-1}(1/2)}{\cos\alpha}\Big)\ge\frac12>\frac14\ge\psi\Big(0,\frac{\varphi^{-1}(1/8)}{\cos\alpha}\Big),$$
hence there exists $y\in(\varphi^{-1}(1/2)/\cos\alpha,\varphi^{-1}(1/8)/\cos\alpha)$ such that $1/4<\psi(0,y)<1/2$ and $\partial_{x_2}\psi(0,y)<0$. Furthermore, since $U$ is symmetric with respect to the axis $\{x_1=0\}$, the function $u_0$ is even in $x_1$, and so are $u(t,\cdot)$ for every $t>0$, and then $\psi$. Thus, $\partial_{x_1}\psi(0,x_2)=0$ for all $x_2\in\R$. From the previous observations, the gradient of $\psi$ at the point $(0,y)$ is a non-zero vector parallel to the vector $(0,1)$. If the function $\psi$ were one-dimensional, it would then necessarily be written as $\psi(x_1,x_2)\equiv\Psi(x_2)$ in $\R^2$, for some $C^2(\R)$ function $\Psi$. In particular, one would have that $\psi(x_1,y)=\psi(0,y)\in(1/4,1/2)$ for all $x_1\in\R$. But~\eqref{vvarphi}-\eqref{ineqV} yield
$$\baa{r}
\displaystyle1\ge\psi(x_1,y)=\lim_{n\to+\infty}u\Big(t_n,x_1,\frac{\zeta(t_n)}{\cos\alpha}+y\Big)\ge\lim_{n\to+\infty}v(t_n,\zeta(t_n)+y\cos\alpha-x_1\sin\alpha)\vspace{3pt}\\
=\varphi(y\cos\alpha-x_1\sin\alpha),\eaa$$
hence $\psi(x_1,y)\to1$ as $x_1\to+\infty$, leading to a contradiction. As a conclusion, the element $\psi$ of $\Omega(u)$ is not one-dimensional. Notice finally that all shifts $\psi(\cdot+a_1,\cdot+a_2)$ of $\psi$ belong to $\Omega(u)$, and are not one-dimensional either.
\end{proof}

\begin{remark}{\rm In the example considered in the above proof, the set $\Omega(u)$ nevertheless contains some elements which are one-dimensional (apart from the constant elements $0$ and $1$, which belong to $\Omega(u)$ by Proposition~\ref{pro:uniformspreading}). Indeed, it follows from~\eqref{vvarphi}-\eqref{ineqV} and $\varphi(+\infty)=0$ that, for any $\varrho:[0,+\infty)\to\R$ with $\varrho(t)\to+\infty$ as $t\to+\infty$, one has
$$u\Big(t,\varrho(t)+x_1,\varrho(t)\tan\alpha+\frac{\zeta(t)}{\cos\alpha}+x_2\Big)\to\varphi(x_2\cos\alpha-x_1\sin\alpha)\ \hbox{ as }t\to+\infty,$$
locally uniformly in $(x_1,x_2)\in\R^2$. Therefore, the one-dimensional non-constant function $(x_1,x_2)\mapsto\varphi(x_2\cos\alpha-x_1\sin\alpha)$ belongs to $\Omega(u)$. So does the one-dimensional non-constant function $(x_1,x_2)\mapsto\varphi(x_2\cos\alpha+x_1\sin\alpha)$, by choosing $\varrho$ such that $\varrho(+\infty)=-\infty$ and adapting the above limit.}
\end{remark}

The second result shows that the conclusions of the main results do not hold in general without the assumption~\eqref{UUdelta}.

\begin{proposition}\label{proV2}
In any dimension $N\ge2$, there are measurable sets $U\subset\R^N$, which satisfy~\eqref{ballcone} but not~\eqref{UUdelta}, such that $\Omega(u)$ contains some elements that are not one-dimensional, where $u$ is the solution of~\eqref{homo} with initial datum $u_0=\1_U$.
\end{proposition}

\begin{proof}
Consider $\beta>0$ and $U:=U_1\cup U_2$, with
\be\label{defU12}
U_1:=\big\{x\in\R^N:x_2\le\beta|x_1|\big\}\ \hbox{ and }\ U_2:=\bigcup_{k\in\Z^N}\overline{B_{e^{-|k|^2}}(k)}.
\ee
The set $U_1$ is the same as in the proof of Proposition~\ref{proV}, and as before we call $\alpha:=\arctan\beta\in(0,\pi/2)$. Notice that $U$ satisfies~\eqref{ballcone} (it is at finite Hausdorff distance from $\R^N$), but not~\eqref{UUdelta}. Let $u$, $u_1$ and $u_2$ denote the solutions of~\eqref{homo} with initial conditions $\1_U$, $\1_{U_1}$ and $\1_{U_2}$, respectively. As in the proof of Proposition~\ref{proV}, one has
\be\label{inequ1u2}\max(u_1(t,x),u_2(t,x))\le u(t,x)\le u_1(t,x)+u_2(t,x)\ \hbox{ for all $t\ge0$ and $x\in\R^N$}.
\ee
Furthermore, since $f(s)\le f'(0)s$ for all $s\ge0$, there holds that
$$0\!\le\!u_2(1,x)\!\le\!\frac{e^{f'(0)}}{(4\pi)^{N/2}}\!\!\sum_{k\in\Z^N}\!\int_{B_{\!e^{-|k|^2}}\!(k)}\!\!\!\!e^{-|x-y|^2/4}dy\le\!\frac{e^{f'(0)-|x|^2/8}}{(4\pi)^{N/2}}\!\!\sum_{k\in\Z^N}\!\int_{B_{\!e^{-|k|^2}}\!(k)}\!\!\!\!e^{|y|^2/4}dy\!\le\! A\,e^{-|x|^2/8}$$
for all $x\in\R^N$, for some positive real number $A$. Therefore, for any $\xi\in\Sph$, one has $0\le u_2(1,x)\le A\,e^{-\sqrt{f'(0)}\,x\cdot\xi+2f'(0)}$ for all $x\in\R^N$, and using again that $f(s)\le f'(0)s$ for all $s\ge0$, the maximum principle implies that
$$0\le u_2(t,x)\le A\,e^{-\sqrt{f'(0)}\,x\cdot\xi+2f'(0)t}\ \hbox{ for all }t\ge1\hbox{ and }x\in\R^N.$$
Since $\xi\in\Sph$ was arbitrary, this means that $0\le u_2(t,x)\le A\,e^{-\sqrt{f'(0)}\,|x|+2f'(0)t}$ for all $t\ge1$ and $x\in\R^N$. In particular, $\sup_{|x|\ge ct}u_2(t,x)\to0$ as $t\to+\infty$, for any $c>c^*=2\sqrt{f'(0)}$.

On the other hand, from the proof of Proposition~\ref{proV} (and the trivial extension of all functions in the variables $(x_3,\cdots,x_N)$), there is a sequence $(t_n)_{n\in\N}$ diverging to $+\infty$ such that the functions $x\mapsto u_1(t_n,x_1,\zeta(t_n)/\cos\alpha+x_2,x_3,\cdots,x_N)$ converge locally uniformly in $\R^N$ to a function $\psi_1\in\Omega(u_1)$ that is not one-dimensional. We also recall that $\zeta(t)\sim c^*t$ as $t\to+\infty$, hence $\zeta(t_n)/\cos\alpha\sim (c^*/\cos\alpha)t_n$ as $n\to+\infty$, with $c^*/\cos\alpha>c^*$. It then follows from~\eqref{inequ1u2} and the conclusion of the previous paragraph that the functions $x\mapsto u(t_n,x_1,\zeta(t_n)/\cos\alpha+x_2,x_3,\cdots,x_N)$ still converge locally uniformly in $\R^N$ to the function $\psi_1$.
Therefore, $\Omega(u)$ contains the element of $\psi_1$, which is not one-dimensional, as well as all its shifts.
\end{proof}

\begin{remark}{\rm The sets $U$ given in Propositions~$\ref{proV}$ and~$\ref{proV2}$ are actually 
closed and equal to the closure of their interior. By doing so, we avoid meaningless counterexamples. For instance, if $U_1$ is as in~\eqref{defU12} and if $U_2=\Z^N$, then $U:=U_1\cup U_2$ satisfies~\eqref{ballcone} (because it is relatively dense in $\R^N$) and it does not satisfy~\eqref{UUdelta}. But the solutions $u$ and $u_1$ of~\eqref{homo} with initial conditions $\1_U$ and $\1_{U_1}$ are actually identical in $(0,+\infty)\times\R^N$ (since the Lebesgue measure of~$U_2$ is equal to $0$), hence $\Omega(u)=\Omega(u_1)$ and this counterexample turns out to be equivalent to the one given in Proposition~$\ref{proV}$.}
\end{remark}


\subsection{Proof of Theorem~\ref{thm:global}}\label{sec55}

Let $u$ be the solution of~\eqref{homo} with an initial condition $u_0=\1_U$ and a set $U$ satisfying~\eqref{UUdelta} and~\eqref{ballcone}. Assume, by way of contradiction, that the conclusion of Theorem~\ref{thm:global} does not hold. Then, there are $k\in\{2,\cdots,N\}$, a sequence $(t_n)_{n\in\N}$ of positive real numbers diverging to~$+\infty$ and a sequence $(x_n)_{n\in\N}$ in $\R^N$, such that
\be\label{liminfsigmai}
\liminf_{n\to+\infty}|\sigma_k(D^2u(t_n,x_n))|>0.
\ee
Up to extraction of a subsequence, the functions $u(t_n,x_n+\cdot)$ converge in $C^2_{loc}(\R^N)$, to an element $\psi$ of $\Omega(u)$. By Theorem~\ref{thm:DGgeneral}, $\psi:\R^N\to\R$ is one-dimensional, hence, for all~$x\in\R^N$, the eigenvalues of $D^2\psi(x)$ are all equal to $0$ except at most one of them, which implies that $\sigma_k(D^2\psi(x))=0$ (because $k\ge2$). On the other hand, since $(-1)^k\sigma_k(D^2u(t_n,x_n))$ is the coefficient of $X^{N-k}$ in the characteristic polynomial $X\mapsto\det(XI_N-D^2u(t_n,x_n))$ (where $I_N$ denotes the identity matrix of size $N\times N$), $\sigma_k(D^2u(t_n,x_n))$ is therefore a polynomial function of the coefficients of $D^2u(t_n,x_n)$. It then follows from the convergence $u(t_n,x_n+\cdot)\to\psi$ in $C^2_{loc}(\R^N)$ that $\sigma_k(D^2u(t_n,x_n))\to\sigma_k(D^2\psi(0))=0$ as $n\to+\infty$. This contradicts~\eqref{liminfsigmai}, and the proof of Theorem~\ref{thm:global} is complete.\hfill$\Box$


\section{Directions of one-dimensional symmetry: proof of Theorem~\ref{thm:E}}\label{sec5}

With Theorem~$\ref{thm:psi}$ in hand, the proof of~\thm{E} consists in showing that $u$ has some partial derivatives which do not vanish as $t\to+\infty$, around suitable sequences of points.
 
\begin{proof}[Proof of Theorem~$\ref{thm:E}$]
The set on the right-hand side of the equivalence stated in the theorem is empty if and only if $U$ is relatively dense in $\R^N$ or $U=\emptyset$. Hence the last statement of the theorem follows from the first one.

Let us show the double inclusion between the sets, as stated in the theorem. The inclusion ``\,$\subset$\,'' is a consequence of Theorem~$\ref{thm:psi}$.

Let us turn to the inclusion ``\,$\supset$\,''. Assume that $U\neq\emptyset$ is not relatively dense in $\R^N$. Let $e\in\Sph$, $\seq{x}$ in $\R^N\setminus\ol U$ and $(\xi_n)_{n\in\N}$ in $\R^N$ be such that
$$\lim_{n\to+\infty}\dist(x_n,U)=+\infty,\ \ \lim_{n\to+\infty}\frac{x_n-\xi_n}{|x_n-\xi_n|}=e,\ \hbox{ and }\ \xi_n\in\pi_{x_n}\hbox{ for all }n\in\N.$$
We need to show that $e\in\mc{E}$. For $n\in\N$, we set for short 
$$k_n:=|x_n-\xi_n|=\dist(x_n,U)\quad\text{ and }\quad e_n:=\frac{x_n-\xi_n}{|x_n-\xi_n|}\in\Sph.$$
We start with showing that 
\Fi{den>}
\liminf_{t\to+\infty}\Big(\,\inf_{n\in\N,\,\lambda\in(0,1)}\partial_{e_n}u(t,\lambda x_n+(1-\lambda)\xi_n)\Big)<0.
\Ff
Assume by contradiction that~\eqref{den>} does not hold. Then, for any $\e>0$ there exists~$\tau_\e>0$ such~that 
\Fi{den<}
\forall\,t\geq \tau_\e,\ \forall\,n\in\N,\ \forall\,\lambda\in(0,1),\quad\partial_{e_n}u(t,\lambda x_n+(1-\lambda)\xi_n)>-\e.
\Ff
Hypothesis~\eqref{UUdelta} implies the existence of two constants~$\delta,R>0$ such that $d_{\mc{H}}(U,U_\delta)<R$. Moreover, by parabolic estimates, there is $K>0$, only depending on $f$ and $N$, such that
\Fi{|Du|<K}
\forall\,t\geq1,\ \forall\,x\in\R^N,\quad |\nabla u(t,x)|\leq K.
\Ff
Call $\delta':=\min(\delta,1/(8K))$. By~\eqref{c<c*} there exists $\tau>0$, only depending on $f$, $N$, $\delta$ and $R$, such that the solution $v$ to~\eqref{homo} with initial datum $v_0=\frac18\1_{B_{\delta'}}$ satisfies $v(t,x)\geq1/2$ for all $t\geq \tau$ and $x\in B_{R}$. We can assume without loss of generality that $\tau\geq1$.

Take $n\in\N$. Because $d_{\mc{H}}(\ol U,U_\delta)=d_{\mc{H}}(U,U_\delta)<R$, there exists $\zeta_n\in U_\delta$ such that~$|\zeta_n-\xi_n|<R$. It follows that $u_0(x)\geq v_0(x-\zeta_n)$ for all~$x\in\R^N$ and therefore, by comparison, 
$$\forall\,t\geq \tau,\quad u(t,\xi_n)\geq v(t,\xi_n-\zeta_n)\geq\frac12.$$
Using~\eqref{den<} we then deduce
$$u\Big(\tau_\e+\tau,\xi_n+\min\Big(\frac1{4\e},k_n\Big)e_n\Big)\geq \frac14.$$
We now start an iterative argument. By~\eqref{|Du|<K} (recall that $\tau\geq1$) we get 
$$\forall\,x\in B_{\frac1{8K}},\quad u\Big(\tau_\e+\tau,\xi_n+\min\Big(\frac1{4\e},k_n\Big)e_n+x\Big)\geq \frac18.$$
Since $\delta'\leq 1/(8K)$, this allows us to compare $u(\tau_\e+\tau+\.,\xi_n+\min(\frac1{4\e},k_n)e_n+\.)$ with $v$ and obtain
$$u\Big(\tau_\e+2\tau,\xi_n+\min\Big(\frac1{4\e},k_n\Big)e_n\Big)\geq \frac12,$$
whence by~\eqref{den<}
$$u\Big(\tau_\e+2\tau,\xi_n+\min\Big(\frac2{4\e},k_n\Big)e_n\Big)\geq \frac14.$$
We iterate $j_n$ times this procedure, where $j_n$ is the smallest $j\in\N$ satisfying $j/(4\e)\geq k_n$. Namely, for any $n\in\N$ we have shown that
\Fi{jn}
u(\tau_\e+j_n\tau,x_n)\geq \frac14\quad\text{with $j_n\in\N$ such that}\; j_n-1<4\e k_n\leq j_n.
\Ff
We compute
$$\frac{\dist(x_n,U)}{\tau_\e+j_n\tau}=\frac{k_n}{\tau_\e+j_n\tau}>\frac{k_n}{\tau_\e+(4\e k_n+1)\tau}\to \frac1{4\e \tau}\as n\to+\infty.$$
Take $c>c^*$ and $\e<1/(4\tau c)$. It follows from the above estimate that $\dist(x_n,U)\!>\!c(\tau_\e+j_n\tau)$ for $n$ large enough, but then~\eqref{jn} contradicts~\eqref{c>c*uniform} because $\seq{j}$ diverges to $+\infty$ since~$\seq{k}$ does. This proves~\eqref{den>}.

We can now conclude. By~\eqref{den>} there exist~$\e>0$, a diverging sequence~$(t_k)_{k\in\N}$ in $\R^+$, a sequence~$(n_k)_{k\in\N}$ in~$\N$ and a sequence $(\lambda_k)_{k\in\N}$ in $(0,1)$ such that
$$\partial_{e_{n_k}}u(t_k,y_k)<-\e,\quad\text{where }\; y_k:=\lambda_k x_{n_k}+(1-\lambda_k)\xi_{n_k}.$$
Therefore, by parabolic estimates, the sequence of functions $(u(t_k,y_k+\.))_{k\in\N}$ converges in~$C^1_{loc}(\R^N)$ (up to subsequences) towards a function $\psi\in\O(u)$ satisfying $\partial_{e}\psi(0)\leq-\e$. Moreover, since $\xi_{n_k}\in\pi_{x_{n_k}}$, there also holds by definition of $y_k$ that $\xi_{n_k}\in\pi_{y_k}$ and
$$\frac{y_k-\xi_{n_k}}{|y_k-\xi_{n_k}|}=e_{n_k}\to e\as k\to+\infty.$$
We deduce from \thm{psi} that $\psi(x)\equiv\Psi(x\.e)$ for some nonincreasing function~$\Psi\in C^2(\R)$. We further know that $\Psi'(0)\leq-\e<0$. Theorem~\ref{thm:psi} then implies that $\Psi$ is actually strictly decreasing in $\R$, hence $e\in\mc{E}$.
\end{proof}


\section{The subgraph case: proof of Corollary~\ref{cor:DGVGM}}\label{sec6}

We now turn to Corollary~\ref{cor:DGVGM}, which is a consequence of Theorems~\ref{thm:DGgeneral} and \ref{thm:E}. In order to check the geometric condition~\eqref{ballcone} of these two theorems, we will make use of the following simple property of functions with vanishing global mean.

\begin{lemma}\label{lem:vgm}
Let $\gamma:\R^{N-1}\to\R$ satisfy~\eqref{VGM}. Then
\be\label{ineqM}
M:=\sup_{{x',y'\in \R^{N-1} }}\frac{|\gamma(x')-\gamma(y')|}{|x'-y'|+1}<+\infty.
\ee
In particular, $|x_N-\gamma(x')|\le M$ for all $(x',x_N)\in\partial U$.
\end{lemma}

\begin{proof}
By~\eqref{VGM}, there exists $L>0$ such that
$$\sup_{x'\in\R^{N-1},\,y'\in \R^{N-1},\,|x'- y'|\geq L}\frac{|\gamma(x')-\gamma(y')|}{|x'-y'|}\leq1.$$
Consider $x',y'\in \R^{N-1}$. Let $z'\in\R^{N-1}$ be such that $|z'-x'|=L$ and $|z'-y'|\ge L$. We have that
$$|\gamma(x')-\gamma(y')|\leq |\gamma(x')-\gamma(z')|+|\gamma(z')-\gamma(y')|\leq L+|z'-y'|\le 2L+|x'-y'|,$$
from which the desired estimate immediately follows. The last statement of Lemma~\ref{lem:vgm}  is an immediate consequence of~\eqref{ineqM}.
\end{proof}

\begin{proof}[Proof of Corollary~$\ref{cor:DGVGM}$]
The proof consists in showing that the assumptions of this corollary entail, on one hand, that $U$ fulfills the hypotheses~\eqref{UUdelta},~\eqref{ballcone} of Theorems~\ref{thm:DGgeneral} and \ref{thm:E}, and, on the other hand, that 
\Fi{e=eN}
\lim_{n\to+\infty}\frac{x_n-\xi_n}{|x_n-\xi_n|}=\mathrm{e}_N,
\Ff
for any sequence $\seq{x}$ satisfying $\dist(x_n,U)\to+\infty$ as $n\to+\infty$ and any sequence $(\xi_n)_{n\in\N}$ such that $\xi_n\in\pi_{x_n}$ for each $n\in\N$. 

By~\eqref{VGM}, there exists $L>0$ such that
\Fi{gamma'>-1}
\forall\, x'\in\R^{N-1},\ \forall\,y'\in\R^{N-1}\!\setminus\!B_L'(x'),\quad\gamma(y')\geq\gamma(x')-|y'-x'|.
\Ff
It follows that, for any $x'\in\R^{N-1}$,
$$U\supset \big\{(y',y_N)\in(\R^{N-1}\!\setminus\!B_L'( x'))\times\R\ :\ y_N\le\gamma( x')-|y'- x'|\big\}.$$
Take $\delta>0$. Since the set in the  right-hand side above contains the $N$-dimensional ball ${B_\delta(( x'+y',\gamma( x')+y_N))}$, for any $(y',y_N)$ with $|y'|=L+\delta$ and $y_N\leq-L-3\delta$, we find~that
$$\partial B_{L+\delta}'( x')\times(-\infty,\gamma( x')-L-3\delta)\subset U_\delta\quad \text{for all $x'\in\R^{N-1}$}.$$
From this inclusion, and the fact that 
$$\forall(x',x_N) \in U,\quad\dist\big((x',x_N)\,,\,\partial B_{L+\delta}'( x')\times(-\infty,\gamma( x')-L-3\delta)\big)\leq\sqrt{(L+\delta)^2+(L+3\delta)^2},$$
we deduce that 
$$d_{\mc{H}}(U,U_\delta)\leq \sqrt{(L+\delta)^2+(L+3\delta)^2},$$
that is,~\eqref{UUdelta} holds for any $\delta>0$.

Next, we claim that
\Fi{Rinfty}
\sup_{x=(x',x_{N})\in\R^N,\,\dist(x,U)=R,\,\xi=(\xi',\xi_{N})\in \pi_{x}}\,\frac{|x'-\xi'|}{R}\to0\ \hbox{ as }R\to+\infty.
\Ff
To show~\eqref{Rinfty}, take $R>0$, consider any point $x_R=(x'_R,x_{R,N})\in\R^N$ such that $\dist(x_R,U)=R$ and let $\xi_R=(\xi'_R,\xi_{R,N})\in \pi_{x_R}$ (remember that $\pi_{x_R}\subset\partial U$). The quantity $h_R:=x_{R,N}-\gamma(x'_R)$ satisfies $h_R\geq R$. We compute
\begin{equation}\label{RhR}
R^2 =|(x_R',\gamma(x_R')+h_R)-(\xi_R',\xi_{R,N})|^2\geq |x_R'-\xi_R'|^2+h_R^2-2h_R|\gamma(x_R')-\xi_{R,N}|.
\end{equation}
If $x_R'-\xi_R'$ stays bounded  as $R\to+\infty$ then the limit in~\eqref{Rinfty} trivially holds. Suppose instead that (up to subsequences) $|x_R'-\xi_R'|\to+\infty$ as $R\to+\infty$. Then, by hypothesis~\eqref{VGM} and Lemma~\ref{lem:vgm}, it follows that $|\gamma(x_R')-\xi_{R,N}|\leq |x_R'-\xi_R'|/2$ for $R$ large, and thus, for such values of $R$, \eqref{RhR} yields
$$R^2 \geq |x_R'-\xi_R'|^2+h_R^2-h_R|x_R'-\xi_R'|\geq\frac12|x_R'-\xi_R'|^2+\frac12 h_R^2\geq\frac12 h_R^2,$$
that is, $h_R\leq\sqrt2 R$. Recalling that $h_R\geq R$, we then derive from~\eqref{RhR} and Lemma~\ref{lem:vgm} that
$$\frac{|x_R'-\xi_R'|}{R}\leq 2 \frac{h_R|\gamma(x_R')-\xi_{R,N}|}{R|x_R'-\xi_R'|}\leq2\sqrt 2 \frac{|\gamma(x_R')-\xi_{R,N}|}{|x_R'-\xi_R'|}\le2\sqrt{2}\,\frac{|\gamma(x'_R)-\gamma(\xi'_R)|+M}{|x_R'-\xi_R'|},$$
which tends to $0$ as $R\to+\infty$ by~\eqref{VGM}. This shows that property~\eqref{Rinfty} holds.

Now, consider any sequence $(x_n)_{n\in\N}$ in $\R^N\setminus\overline{U}$ such that $\dist(x_n,U)\to+\infty$ as $n\to+\infty$, and any sequence $(\xi_n)_{n\in\N}$ such that $\xi_n\in\pi_{x_n}$ for each $n\in\N$. By Lemma~\ref{lem:vgm}, one has that
$$|x_n-\xi_n|\ge x_{n,N}-\xi_{n,N}\ge x_{n,N}-\gamma(\xi'_n)-M\ge x_{n,N}-\gamma(x'_n)-2M-M|x'_n-\xi'_n|$$
for all $n\in\N$. But $x_{n,N}-\gamma(x'_n)=|x_n-(x'_n,\gamma(x'_n))|\ge|x_n-\xi_n|$ since $\xi_n\in\pi_{x_n}$. Hence $|x_n-\xi_n|\ge x_{n,N}-\xi_{n,N}\ge|x_n-\xi_n|-2M-M|x'_n-\xi'_n|$, and the last quantity is equivalent to $|x_n-\xi_n|=\dist(x_n,U)$ as $n\to+\infty$, by~\eqref{Rinfty}. As a consequence, $x_{n,N}-\xi_{n,N}\sim|x_n-\xi_n|$ as $n\to+\infty$ and, together with~\eqref{Rinfty} again, the property~\eqref{e=eN} follows.

We are left to show~\eqref{ballcone}. As before, let $x_R=(x'_R,x_{R,N})\in\R^N$ such that $\dist(x_R,U)=R>0$. We want to estimate~$\mc{O}(x_R)$ defined by~\eqref{opening} when $R$ is large, i.e.
$$\mc{O}(x_R)=\sup_{\xi\in\pi_{x_R},\,y\in U\setminus\{\xi\}}\,\frac{x_R-\xi}{R}\.\frac{y-\xi}{|y-\xi|}.$$
We first consider the set of points $\xi,y$ satisfying $|y-\xi|<\sqrt R$. Since at any $\xi\in\pi_{x_R}$ and~$y\in U$ it holds that
$$R^2\leq |x_R-y|^2=R^2+|\xi-y|^2+2(x_R-\xi)\.(\xi-y),$$
we derive
\Fi{<sqrtR}
\sup_{\xi\in\pi_{x_R} ,\,y\in U,\,0<|y-\xi|<\sqrt R}\,\frac{x_R-\xi}R\.\frac{y-\xi}{|y-\xi|}\leq \frac{1}{2\sqrt R}\to0 \as R\to+\infty.
\Ff
It remains to estimate the above scalar product when $|y-\xi|\geq\sqrt R$. We first observe that, for $\xi=(\xi',\xi_N')\neq y=(y',y_N)$, 
\Fi{O<}
\frac{x_R-\xi}{R}\.\frac{y-\xi}{|y-\xi|}\leq\left(\frac{|x_R'-\xi'|}R+\frac{(x_{R,N}-\xi_{N})(y_N-\xi_{N})}{R|y-\xi|}\right).
\Ff
Let $\xi\in\pi_{x_R}$ and~$y\in U\setminus\{\xi\}$. The first term of the right-hand side is handled by~\eqref{Rinfty}. As for the second term, we notice that~\eqref{e=eN} and Lemma~\ref{lem:vgm} imply that $x_{R,N}-\xi_N\ge0$ for $R$ large enough, and $\xi_N\ge\gamma(\xi')-M$. Therefore, since $y_N\leq \gamma(y')$, it follows that
\be\label{xRN}
\frac{(x_{R,N}-\xi_{N})(y_N-\xi_{N})}{R|y-\xi|}\leq\frac{(x_{R,N}-\xi_{N})(\gamma(y')-\gamma(\xi')+M)}{R|y-\xi|}\leq\frac{|\gamma(y')-\gamma(\xi')|+M}{|y-\xi|}
\ee
for all $R$ large enough. From this, on one hand, restricting to~$|y'-\xi'|\geq \sqrt[3]{R}$, it follows from~\eqref{VGM} that
\Fi{>sqrt3R}
\sup_{\substack{\dist(x_R,U)=R,\,\xi=(\xi',\xi_N')\in\pi_{x_R}\\ y=(y',y_N)\in U,\,|y'-\xi'|\geq \sqrt[3]{R}}}\!\!\!\frac{(x_{R,N}\!-\!\xi_{N})(y_N\!-\!\xi_{N})}{R|y-\xi|}\leq\!\sup_{\substack{\xi',y'\in\R^{N-1} \\ |y'-\xi'|\geq \sqrt[3]{R}}}\!\!\!\frac{|\gamma(y')\!-\!\gamma(\xi')|\!+\!M}{|y'-\xi'|}\!\mathop{\longrightarrow}_{R\to+\infty}0.
\Ff
On the other hand, when $|y'-\xi'|< \sqrt[3]{R}$, we deduce from~\eqref{xRN} that
\Fi{<sqrt3R}
\sup_{\substack{\dist(x_R,U)=R,\,\xi=(\xi',\xi_N')\in\pi_{x_R}\\ y=(y',y_N)\in U\,|y'-\xi'|< \sqrt[3]{R}, \ |y-\xi|\geq\sqrt{R}}}\frac{(x_{R,N}-\xi_{N})(y_N-\xi_{N})}{R|y-\xi|}\leq\frac{2M+\sqrt[3]{R}}{\sqrt{R}}\mathop{\longrightarrow}_{R\to+\infty}0.
\Ff
Summing up,~\eqref{ballcone} follows from the estimates~\eqref{<sqrtR} and~\eqref{O<}, \eqref{Rinfty},~\eqref{>sqrt3R}, and~\eqref{<sqrt3R}.
\end{proof} 


\section{Directional asymptotic one-dimensional symmetry}\label{sec7}

The arguments of the proof of \thm{DGgeneral} can somehow be localized. Loosely speaking, if one focuses on the asymptotic one-dimensional property around a given direction, the global geometric assumption~\eqref{ballcone} can be restricted to the points $x$ around that direction, and hypothesis~\eqref{UUdelta} can be relaxed too. Under such weaker assumptions, we derive the one-dimensional symmetry for functions belonging to the {\em directional $\O$-limit set} of the solution, which is defined as follows.

\begin{definition}\label{def:Omegae}	
For a given bounded function $u:\R^+\times\R^N\to\R$ and for any direction $e\in\Sph$, the set	
$$\begin{array}{ll}
\Omega_e(u):=\big\{&\!\!\!\!\psi\in L^\infty(\R^N)\ :\   u(t_n,x_n+\.)\to\psi\text{ in $L^\infty_{loc}(\R^N)$}\\
& \!\!\!\! \text{for some sequences $(t_n)_{n\in\N}$ in $\R^+$ diverging to $+\infty$},\\
& \!\!\!\! \text{and $(x_n)_{n\in\N}$ in $\R^N\setminus\{0\}$ such that $x_n/|x_n|\to e$ as $n\to+\infty$}\,\big\}\end{array}$$
is called the {\em $\O$-limit set  in the direction $e$} of $u$. Notice that $\Omega_e(u)\subset\Omega(u)$.
\end{definition}

\begin{theorem}\label{thm:DGe}
Let $u$ be a solution of~\eqref{homo} with an initial condition  $u_0\!=\!\1_U$, where $U\!\subset\!\R^N$ has nonempty interior and satisfies
\Fi{Ucontained}
U\subset\big\{(x',x_N)\in\R^{N-1}\times\R \ :\ x_N\leq\gamma(x')\big\},
\Ff
for a function $\gamma\in L^\infty_{loc}(\R^{N-1})$ such that
\Fi{gamma<0}
\limsup_{|x'|\to+\infty}\frac{\gamma(x')}{|x'|}\leq0.
\Ff
Then, any function $\psi\in\Omega_{\mathrm{e}_N}(u)$ is one-dimensional and satisfies $\psi(x',x_N)\equiv\Psi(x_N)$ in~$\R^N$, 
with $\Psi\in C^2(\R)$ either constant or strictly decreasing. In particular it holds that
$$\nabla_{\!x'}u(t,x',x_N)\to0\as t\to+\infty,\text{ locally in $x'\!\in\!\R^{N-1}$ and uniformly in $x_N\!\in\![R,+\infty)$},$$
for any $R\in\R$, and moreover if the inclusion is replaced by an equality in~\eqref{Ucontained}, then
$$\nabla_{\!x'}u(t,x',x_N)\to0\as t\to+\infty,\text{ locally in $x'\!\in\!\R^{N-1}$ and uniformly in $x_N\!\in\!\R$}.$$
\end{theorem}

\begin{proof}	
We prove the result showing that, when restricted to the directional $\O$-limit set~$\O_{\mathrm{e}_N}(u)$, the arguments of the proof of Theorem~\ref{thm:psi} can be performed with hypotheses~\eqref{UUdelta} and~\eqref{ballcone} replaced by the assumptions  that $U$ has nonempty interior and fulfills~\eqref{Ucontained}-\eqref{gamma<0}. We will also show that the functions in $\O_{\mathrm{e}_N}(u)$ are one-dimensional precisely in the direction~$\mathrm{e}_N$. The	situation is simpler here and we do not need to introduce any coordinates transformation.

Assume by contradiction that there exists $\psi\in\O_{\mathrm{e}_N}(u)$ satisfying $\nabla_{\!x'}\psi(\bar x)\neq0$ for some~$\bar x\in\R^N$. Let $(t_n)_{n\in\N}$ in $\R^+$ and $(x_n)_{n\in\N}$ in $\R^N\!\setminus\!\{0\}$ be the associated sequences given in Definition~\ref{def:Omegae}, that is,
\Fi{tnxn}
t_n\to+\infty\quad\text{ and }\quad \frac{x_n}{|x_n|}\to \mathrm{e}_N\as n\to+\infty.
\Ff
Since~$U$ has nonempty interior, the invasion property~\eqref{c<c*} applies and yields
\Fi{|xn|>}
\forall\, c\in(0,c^*),\quad|x_n|\geq c t_n\quad \text{for all $n$ sufficiently large},
\Ff
because otherwise $\psi\equiv1$ in $\R^N$. In particular, $\seq{x}$ needs to be~unbounded. Therefore, up to replacing $\seq{x}$ with $(x_n+\bar x)_{n\in\N}$, we can assume without loss of generality that $\bar x=0$. Namely, the sequences $(t_n)_{n\in\N}$, $(x_n)_{n\in\N}$ satisfy~\eqref{tnxn} and~\eqref{|xn|>}, and by parabolic estimates it holds that
\Fi{Dx'u}
\nabla u(t_n,x_n)\to\beta\as n\to+\infty,
\Ff
with $\beta=(\beta',\beta_N)$, $\beta'\neq0$. We write, for $n\in\N$, $x_n=(x_n',x_{n,N})$. Properties~\eqref{tnxn},~\eqref{|xn|>} immediately imply 
\Fi{xNn>}
\forall\, c\in(0,c^*),\quad x_{n,N}=x_n\.\mathrm{e}_N\geq c t_n\quad \text{for all $n$ sufficiently large}.
\Ff	

Similarly,~\eqref{c>c*uniform} implies that $\limsup_{n\to+\infty}\dist(x_n,U)/t_n\le c^*$ (because otherwise $\psi\equiv0$ in $\R^N$). Furthermore, it follows from~\eqref{Ucontained}-\eqref{tnxn} and $\lim_{n\to+\infty}|x_n|=+\infty$, that $\dist(x_n,U)\sim x_{n,N}$ as $n\to+\infty$. Therefore, $\limsup_{n\to+\infty}x_{n,N}/t_n\le c^*$, hence
\be\label{x'ntn}
|x'_n|=o(t_n)\ \hbox{ as }n\to+\infty,
\ee
by using~\eqref{tnxn} again.

For given $\vt>0$, define the sets
$$H_n:=\ol{B'_{\vt t_n}}\times(-\infty,\vt t_n],\qquad V_n:=\big\{x_n+s(\beta+\zeta)\ :\ s\in\R,\,\zeta\in B_\vt\big\}.$$
Consider a point $P=(P',P_N)\in V_n$ satisfying $P_N\leq\vt t_n$. Namely, $P=x_n+s(\beta+\zeta)$ for some $\zeta\in B_\vt$ and $s\in\R$ such that
$$\vt t_n\geq x_{n,N}-(|\beta_N|+\vt) |s|.$$
As a consequence, if $\vt<|\beta'|$ we derive
$$|P'|\geq (|\beta'|-\vt)|s|-|x_n'|\geq \frac{|\beta'|-\vt}{|\beta_N|+\vt}(x_{N,n}-\vt t_n)-|x_n'|.$$
Since $\beta'\neq0$ and $|x'_n|/x_{n,N}\to 0$ as $n\to+\infty$ due to~\eqref{tnxn}, using~\eqref{xNn>} one can find $\vt\in(0,\min(|\beta'|,c^*/2))$ sufficiently small, only depending on $\beta'$, $\beta_N$ and~$c^*$, such that, for $n$ large, $|P'|>\vt t_n$, i.e.~$P\notin H_n$. With this choice it holds that $H_n\cap V_n=\emptyset$ for $n$ sufficiently large. 

We then set
\be\label{defL2_e}
\sigma:=\frac\vt3>0,\qquad\e:=\frac\vt2>0.
\ee
We further take $\delta>0$ such that $U_\delta\neq\emptyset$ and finally $L>0$ large enough so that~\eqref{assapprox} holds, the latter being possible due to~\eqref{Ucontained}-\eqref{gamma<0}. Observe that $0<\sigma<\vt<c^*/2$. This means that~$U$ fulfills the hypotheses of Lemma \ref{lem:approximation}. Therefore, for $n\in\N$, the solution~$u^{\vt t_n}$ of~\eqref{homo} whose initial datum is equal to the indicator function of the set $U\cap(B_{\vt t_n}'\times\R)$ satisfies~\eqref{utaue},~i.e.
$$\big\|u(t_n,\.)-u^{\vt t_n}(t_n,\.)\big\|_{C^1(B'_{\vt t_n/3}\times\R^+)}<\frac\vt2,$$
provided $n$ is sufficiently large. By~\eqref{Dx'u}-\eqref{x'ntn}, one then derives $|\nabla u^{\vt t_n}(t_n,x_n)-\beta|<\vt$ for all $n$ sufficiently large. This means that, for such values of $n$, the line $\Gamma_n$ passing through the point $x_n$ and directed as $\nabla u^{\vt t_n}(t_n,x_n)$ is contained in the set $V_n$ defined before,~whence
$$\big(\Gamma_n\cap H_n\big)\subset \big(V_n \cap H_n\big)=\emptyset\quad \text{for all $n$ sufficiently large}.$$
On the other hand, we have that 
\Fi{suppu0n2}
\supp u^{\vt t_n}(0,\.)\subset\overline{U}\cap(\overline{B_{\vt t_n}'}\times\R)\subset H_n\quad \text{for all $n$ sufficiently large}.
\Ff
We are thus in a position to apply the reflection argument. Namely, proceeding as in Step~4 of the proof of Theorem~\ref{thm:psi}, with $u_n^{\vt t_n}$ replaced by $u^{\vt t_n}$, we reach a contradiction thanks to the Hopf lemma. This proves that any~$\psi\in\O_{\mathrm{e}_N}(u)$ satisfies $\nabla_{\!x'}\psi\equiv0$, that is, there is a $C^2(\R)$ function $\Psi$ such that $\psi(x)\equiv\Psi(x_N)$ in $\R^N$. The monotonicity of $\Psi$ can then be shown as in Step~5 of the proof of Theorem~\ref{thm:psi}, with $u_n$ replaced by $u$, $k_n\mathrm{e}_N$ by~$x_n$, and~$k_n$ by~$x_{n,N}$. Furthermore, similarly as in Step~6 of the proof of Theorem~\ref{thm:psi}, the function~$\Psi$ is either constant or strictly decreasing.

Let us deal now with the last part of the theorem concerning the convergences of $\nabla_{\!x'}u$ towards~$0$. Consider a diverging sequence $(t_n)_{n\in\N}$ in~$\R^+$, a bounded sequence $(x_n')_{n\in\N}$ in~$\R^{N-1}$ and a sequence~$(x_{n,N})_{n\in\N}$ in $\R$. By parabolic estimates, as $n\to+\infty$, the function $u(t_n,(x_n',x_{n,N})+\.)$ converge in $C^2_{loc}(\R^N)$, up to extraction of a subsequence, towards a function $\psi$. On one hand, if up to extraction of another subsequence, $x_{n,N}\to+\infty$ as~$n\to+\infty$, then $\psi\in\Omega_{\mathrm{e}_N}(u)$ and thus, by the first part of the theorem proved above, there is a $C^2(\R)$ function $\Psi$ such that $\psi(x)\equiv\Psi(x_N)$ in $\R^N$. On the other hand, if $(x_{n,N})_{n\in\N}$ is bounded then $\psi\equiv1$ in $\R^N$ due to~\eqref{c<c*}. Summing up, we have~$\nabla_{\!x'}u(t_n,x_n',x_{n,N})\to0$ as~$n\to+\infty$ when $(x_{n,N})_{n\in\N}$ is bounded from below. This proves the first convergence of $\nabla_{\!x'}u$ stated in the theorem. We are left with the case where,  up to subsequences, $x_{n,N}\to-\infty$ as~$n\to+\infty$ and the inclusion is replaced by an equality in~\eqref{Ucontained}. In such a case, even if it means repla\-cing $U$ by a measurable set $U'\supset U$ such that $U'\setminus U$ has zero Lebesgue measure, the set $U$ contains, for given $\delta>0$, the half-cylinder $B_\delta'\times(-\infty,\essinf_{B_\delta'}\gamma)$, where $\essinf_{B_\delta'}\gamma>-\infty$ because~$\gamma\in L^\infty_{loc}(\R^{N-1})$. We deduce
$$U_\delta\supset \{0\}\times(-\infty,\essinf_{B_\delta'}\gamma-\delta),$$ 	
where $0$ above stands for the origin in $\R^{N-1}$. It follows from property~\eqref{c<c*uniform} of Proposition~\ref{pro:uniformspreading} that $u(t,x',x_N)\to1$ as $t\to+\infty$ locally in $x'\!\in\!\R^{N-1}$ and uniformly in $x_N\!\in\!(-\infty,R]$, for any $R>0$. This implies that $\psi\equiv1$ in $\R^N$, hence~$\nabla_{\!x'}u(t_n,x_n',x_{n,N})\to0$ as~$n\to+\infty$ up to subsequences. The proof of the theorem is complete.
\end{proof}


\section{Extensions and open questions}\label{secgeneral}

We list in this last section some extensions of the main results for more general initial data, as well as some open questions and conjectures.

\subsubsection*{Some extensions of the main results}

First of all, we point out that the conclusions of the main results of Section~\ref{secaos} still hold for the solutions to~\eqref{homo} with measurable initial conditions $u_0:\R^N\to[0,1]$ more general than characteristic functions. To be more precise, if there are $h\in(0,1]$ and $\delta>0$ such that~\eqref{UUdelta} is replaced~by
\be\label{assh1}
d_{\mc{H}}\big(\{u_0\ge h\},\supp u_0\big)<+\infty\ \hbox{ and }\ d_{\mc{H}}\big(\{u_0\ge h\},\{u_0\ge h\}_\delta\big)<+\infty,
\ee
and if~\eqref{ballcone} is replaced by
\be\label{assh2}
\lim_{R\to+\infty}\bigg(\,\sup_{x\in\R^N,\,\dist(x,\supp u_0)=R}\mc{O}(x)\bigg)\leq 0,
\ee
then the conclusion of \thm{psi} --~and thus of Theorem~$\ref{thm:DGgeneral}$~-- holds~true. Indeed, first of all, one checks that Proposition~\ref{pro:uniformspreading} holds with $U$ and $U_\delta$ replaced by $\supp u_0$ and~$\{u_0\ge h\}_\delta$ respectively, where, for the proof of~\eqref{c<c*uniform}, one defines $v$ as the solution with initial datum $v_0= h\1_{B_\delta}$. Then Lemma~$\ref{lem:approximation}$ still holds with $u_0^R:=u_0\,\1_{B'_R\times\R}$ and with the assumption~\eqref{assapprox} replaced~by
$$\{u_0\ge h\}_\delta\cap B_L\neq\emptyset \ \text{ \ and \ }\ 
\supp u_0 \setminus (B'_L\times\R)\,\subset\,\Big\{(x',x_N)\in\R^{N-1}\times\R \ : \ x_N\leq\frac{\sigma}{2c^*}|x'|\Big\}$$
(but now in the conclusion~\eqref{utaue} the time $\tau_\epsilon$ depends on $h$ too). Next, one repeats the arguments of the proof of \thm{psi} with the $U_n$ defined as rigid transformations of $\supp u_0$ in place of $U$, and $L:=d_{\mc{H}}\big(\{u_0\ge h\},\{u_0\ge h\}_\delta\big)+1$ in~\eqref{defL2} and~\eqref{sigmaL2}.

As a consequence of Theorem~\ref{thm:DGgeneral}, the conclusion of Theorem~\ref{thm:global} still holds for initial conditions $u_0$ satisfying~\eqref{assh1}-\eqref{assh2} instead of $u_0=\1_U$ with~\eqref{UUdelta} and~\eqref{ballcone}. Similarly, the conclusion of Theorem~\ref{thm:E} holds for such $u_0$'s, with $U$ replaced by $\supp u_0$ in the statement, while the conclusion of Corollary~\ref{cor:DGVGM} holds when $d_{\mc{H}}\big(\{u_0\ge h\},\supp u_0\big)<+\infty$ and $d_{\mc{H}}\big(U,\supp u_0\big)<+\infty$, with $U$ still satisfying~\eqref{Usub}-\eqref{VGM}. Finally, the conclusion of Theorem~\ref{thm:DG} is satisfied when $u_0$ fulfills~\eqref{assh1} instead of~\eqref{defu0} and~\eqref{UUdelta}, and when the convexity --~or convex proximity~-- of $U$ is replaced by that of~$\supp u_0$.

\subsubsection*{Some open questions and conjectures}

To complete the paper, we propose a list of open questions and conjectures related to our results. First of all, let us call $\varphi$ the traveling front profile with minimal speed, that is, for each $e\in\Sph$, $\varphi(x\cdot e-c^*t)$ satisfies~\eqref{homo} with $0=\varphi(+\infty)<\varphi<\varphi(-\infty)=1$ and $c^*=2\sqrt{f'(0)}$. Based on Theorems~\ref{thm:DGgeneral} and~\ref{thm:psi}, and according to the definition~\eqref{E} of $\mc{E}$, we propose the following.

\begin{conjecture}
Let $u$ be as in Theorem~$\ref{thm:DGgeneral}$. Then,
\[
\O(u)=\big\{\, 0,\ 1,\ \vp(x\.e +a)\ : \ e\in\mc{E},\ a\in\R\,\big\}.
\]
\end{conjecture}

This conjecture is known to hold when $U$ is bounded with non-empty interior, by~\cite{D,RRR}, and when $U$ is the subgraph of a bounded function, or more generally when there are two half-spaces $H$ and $H'$ --necessarily with parallel boundaries-- such that $H\subset U\subset H'$, by~\cite{BH1,B,HNRR,L,U1}.

\medskip

We have shown in Lemma~\ref{lem:ballcone} that the assumption~\eqref{ballcone} of Theorem~\ref{thm:DGgeneral} is stable by bounded perturbations of the sets $U$. We could then wonder whether the \aos\ is also stable with respect to bounded perturbations of the initial support. Namely, if the solution to~\eqref{homo} with an initial datum $\1_{U}$ satisfying~\eqref{UUdelta} is asymptotically locally planar, and if $U'\subset\R^N$ satisfies~\eqref{UUdelta} and $d_{\mc{H}}(U',U)<+\infty$, then is the solution to~\eqref{homo} with initial datum $\1_{U'}$ asymptotically locally planar as well~?

\medskip

One can also wonder whether the reciprocal of Theorem~\ref{thm:DGgeneral} is true, in the following sense: if the \aos\ holds for a solution $u$ of~\eqref{homo} with initial datum $\1_{U}$ and $U$ satisfying~\eqref{UUdelta}, does necessarily $U$ fulfill~\eqref{ballcone}~? The answer is immediately seen to be negative in general: take for instance $U$ given by 
$$U=\bigcup_{n\in\N}[2^n,2^n+1]\times\R^{N-1},$$
which fulfills~\eqref{UUdelta} but not~\eqref{ballcone}, while $u$ --~hence any element of $\Omega(u)$~-- is one-dimensional, depending on the variable $x_1$ only. However, the question is open if~$U$ is connected.

\medskip

Our study concerns the Fisher-KPP equation, with functions $f$ satisfying~\eqref{fkpp}. However, the same question of \aos\ can be asked for more general reaction terms $f$, still with $f(0)=f(1)=0$. First of all, the hypothesis~\eqref{UUdelta} should be strengthened, by requiring $\delta>0$ to be large enough. Indeed, if $f$ is for instance of the bistable type
\be\label{bistable}
f'(0)<0,\ \ f'(1)<0,\ \ f<0\hbox{ in }(0,\alpha),\ \ f>0\hbox{ in }(\alpha,1),\ \ \int_0^1f(s)ds>0
\ee
for some $\alpha\in(0,1)$, then by~\cite{DM1,Z} there is $\delta_0>0$ such that the solution to~\eqref{homo} with initial condition $u_0=\1_{B_{\delta_0}}$ converges uniformly as $t\to+\infty$ to a ground state, that is, a positive radial solution converging to $0$ as $|x|\to+\infty$, hence $u$ is not asymptotically locally planar. However, if $u_0:=\1_{B_\delta}$ with $\delta>\delta_0$, then $u(t,x)\to1$ as $t\to+\infty$ locally uniformly in $x\in\R^N$. We then say that the {\em invasion property} holds if there is $\rho$ such that the solution $u$ to~\eqref{homo} with initial condition $\1_{B_{\rho}}$ satisfies $u(t,x)\to1$ as $t\to+\infty$ locally uniformly in $x\in\R^N$. For general functions $f$ for which the invasion property holds, if $U$ is bounded and $U_{\rho}\neq\emptyset$, then the solutions to~\eqref{homo} with initial condition $\1_U$ are known to be asymptotically locally planar, by~\cite{J}. The same conclusion holds for bistable functions $f$ of the type~\eqref{bistable} if there are two half-spaces $H$ and $H'$ --necessarily with parallel boundaries-- such that $H\subset U\subset H'$, by~\cite{BH1,FM,MN,MNT} (see also~\cite{P2} for the case of more general functions $f$). On the other hand, still for bistable functions $f$ of the type~\eqref{bistable}, the solutions $u$ to~\eqref{homo} with initial condition $\1_U$ are not asymptotically locally planar if $U$ is V-shaped, that is, if it is the union of two half-spaces with non-parallel boundaries, by~\cite{HMR1,HMR2,NT,RR1}. These known results lead us to formulate the following De Giorgi type conjecture for the solutions of the reaction-diffusion equation~\eqref{homo} beyond the Fisher-KPP case. 

\begin{conjecture}\label{conj:DG}
Assume that the invasion property holds for some $\rho>0$. Let $u$ be the solution to~\eqref{homo} with an initial datum $u_0=\1_{U}$ such that $U\subset\R^N$ satisfies~$d_{\mc{H}}(U,U_\rho)<+\infty$ and~\eqref{ballcone}. Then any function in $\Omega(u)$ is one-dimensional and, in addition, it is either constant or strictly monotone.
\end{conjecture}

Let us also mention another natural question related to the preservation of the convexity of the upper level sets of $u$ when $u_0=\1_U$ and $U$ is convex. It is known from~\cite{BL,IST} that, if $U$ is convex, then the solution of the heat equation $\partial_tu=\Delta u$ is quasi-concave at each $t>0$, that is, for each $t>0$ and $\lambda\in\R$, the upper level set $\{x\in\R^N:u(t,x)>\lambda\}$ is convex. The same conclusion holds for~\eqref{homo} set in bounded convex domains instead of $\R^N$, and under some additional assumptions on $f$, by~\cite{IS}. A natural question is to wonder for which class of functions $f$ this property still holds for~\eqref{homo} in $\R^N$.

\medskip

Notice finally that, for any solution $u$ to~\eqref{homo}, for any sequence $(t_n)_{n\in\N}$ diverging to~$+\infty$, and for any sequence $(x_n)_{n\in\N}$ in $\R^N$, the functions $u(t_n+\cdot,x_n+\cdot)$ converge locally uniformly in $\R\times\R^N$, up to extraction of a subsequence, to an entire solution to~\eqref{homo} (that~is,~solution~for~all~$t\in\R$). Remembering Theorem~\ref{thm:DG} on the \aos\ for the solutions to~\eqref{homo} with $u_0=\1_U$ and $U$ convex, and having in mind the question of the previous paragraph on the convexity of the upper level sets, it is then natural to ask the following: if an entire solution $v:\R\times\R^N\to[0,1]$ to~\eqref{homo} is quasi-concave for every $t\in\R$, is $v(t,\cdot)$ necessarily one-dimensional for every $t\in\R$~?


\end{document}